\documentclass[11pt,reqno]{amsart}
\usepackage[tmargin=1in,bmargin=1in,rmargin=1in,lmargin=1in]{geometry}
\usepackage[mathscr]{eucal}
\usepackage[breaklinks=true]{hyperref}

\usepackage{longtable}
\usepackage{booktabs}
\usepackage{tabularx}
\usepackage{makecell}
\usepackage{xcolor}

\usepackage{amsmath,amsthm,amsfonts,amssymb,tocvsec2}
\usepackage{mathdots}
\usepackage{mathtools}
\usepackage{caption}
\usepackage{subcaption}
\usepackage{float}
\usepackage{multicol}
\usepackage{enumitem}
\usepackage{rotating}
\usepackage{multirow}

\theoremstyle{plain}
\newtheorem{theorem}{Theorem}[section]

\newtheorem{lemma}[theorem]{Lemma}

\newtheorem{cor}[theorem]{Corollary}

\theoremstyle{definition}
\newtheorem{defn}[theorem]{Definition}

\newtheorem{rem}[theorem]{Remark}
\newtheorem{example}[theorem]{Example}

\numberwithin{equation}{section}

\DeclareMathOperator{\rk}{rank}
\DeclareMathOperator{\cpr}{cpr}

\DeclareMathOperator{\diag}{diag}

\begin{document}
	\title[Linear preserver problems in matrix positivity theory]{Linear preserver problems in matrix positivity theory}

	\author{Projesh Nath Choudhury and Shivangi Yadav}
	\address[P.N.~Choudhury]{Department of Mathematics, Indian Institute of Technology Gandhinagar, Gujarat 382355, India}
	\email{\tt projeshnc@iitgn.ac.in} 
	\address[S.~Yadav]{Department of Mathematics, Indian Institute of Technology Gandhinagar, Gujarat 382355, India}
	\email{\tt shivangi.yadav@iitgn.ac.in, shivangi97.y@gmail.com}

	\date{\today}
	
	\begin{abstract}
		Linear preserver problems have been a central focus of research in matrix theory and operator theory for more than a century, beginning with Frobenius' 1897 characterization of determinant preserving linear maps on the space of complex matrices. Since this foundational result, considerable work has examined linear preservers of diverse subsets, functions, and relations across different matrix and operator spaces. The purpose of this survey is to present the current state of research on linear preserver problems for several positivity classes of matrices. We provide an overview of recent developments in the literature and, for each positivity class considered, identify gaps that remain to guide future research.
	\end{abstract}
	
	
	
	\maketitle
	
	\vspace*{-11mm}
	\settocdepth{subsection}
	\tableofcontents
	
	\section{Introduction}
	The study of linear maps on spaces of bounded linear operators with special properties such as leaving certain subsets, functions, or relations invariant is a rich and active area of research in matrix theory and operator theory. The origin of preserver problems is attributed to Frobenius~\cite{Frobenius1897}, who posed the fundamental question: which linear maps on the space of $n\times n$ complex matrices preserve the determinant? It is straightforward to verify that all maps of the form
    \begin{align}\label{det-preserver}
    	A \mapsto MAN \quad \text{or} \quad A \mapsto MA^TN,
    \end{align} 
    where $M,N\in\mathbb{C}^{n\times n}$ are invertible matrices satisfying $\det (MN)=1$, preserve the determinant of all $A\in\mathbb{C}^{n\times n}$. In 1897, Frobenius showed that, in fact, every determinant preserver must be of the form~\eqref{det-preserver}.
    
    Subsequent work extended Frobenius' result in various directions. For example, Eaten~\cite{Eaten} studied linear maps on the space of symmetric matrices that scale the determinant by some fixed non-zero real constant. Similar results were obtained by Dolinar--\v{S}emrl~\cite{Dolinar} (on the space of complex square matrices), Tan--Wang~\cite{Tan} (on the space of square matrices over any field $\mathbb{F}$, as well as on the space of upper triangular matrices over an arbitrary field $\mathbb{F}$), and Cao--Tang~\cite{Cao} (on the space of symmetric matrices over any field $\mathbb{F}$). In all these cases, the authors characterized linear maps $\phi$ satisfying
    \begin{align*}
    	\det(\alpha \phi(A) + \phi(B)) = \det(\alpha A + B), ~ \text{for all} ~ \alpha\in\mathbb{F}.
    \end{align*}
    Recently, Nagy~\cite{Nagy17} established an infinite-dimensional analogue of Frobenius' result. In addition to the study of determinant preservers, numerous works have examined linear maps that preserve other structural properties. 
    
    The beauty of linear preserver problems lies in the simplicity of their formulation and the elegance of their solutions. Moreover, they are also important because they often simplify otherwise difficult questions. For instance, computing the permanent of a matrix is significantly more challenging than computing its determinant, especially for higher order matrices. This difficulty motivated P\'olya~\cite{Polya} to ask whether there exist linear transformations that convert the permanent of a matrix into its determinant; that is, classify all linear maps
	\begin{align*}
		\phi:\mathbb{R}^{n\times n}\to\mathbb{R}^{n\times n} \quad \text{such that} \quad \mathrm{per}\phi(A) = \det(A).
	\end{align*}
	If such maps existed, the computation of permanents could be greatly simplified via the linear transformation $\phi$. However, Marcus--Minc~\cite{Marcus_etal} showed that no such linear operator exists for $n\geq3$. Other examples arise in the investigation of linear maps that transform quantum systems while preserving entropy (see, for instance, \cite{Fosner_et_al_survey}), as well as in the characterization of linear operators on matrix spaces that preserve controllable matrix pairs (see~\cite{Fung96} and the references therein). A further example concerns the use of linear transformations to simplify the analysis of systems of differential equations, particularly those linear maps that preserve eigenmodes or stability of the system. Addressing such problems provides mechanisms for reducing complex systems to more tractable forms, while ensuring that their essential structural properties are preserved. Collectively, these problems fall within the framework of linear preserver problems, underscoring both its breadth and its significance.
	
	Linear preserver problems are also widely studied within matrix positivity theory. Positivity classes of matrices are objects of central interest in mathematics, as they arise naturally in diverse areas, e.g., positive semidefinite matrices have applications in differential equations~\cite{Hahn}, numerical analysis~\cite{Trefethen-Bau}, and statistics~\cite{CR_Rao75}; copositive matrices are well studied in optimization theory~\cite{Dur} and matrix theory~\cite{Monderer-Berman}; semipositive matrices appear in different problems, such as the characterization of M-matrices, linear complementarity problems, and economic models~\cite{Berman_Plemmons}; and totally nonnegative matrices and kernels~\cite{K68} arise in probability and statistics~\cite{Karlin64}, particle systems~\cite{GK37,GK50}, representation theory~\cite{Lu94, Ri03}, and cluster algebras / totally nonnegative Grassmannian~\cite{FZ02,Galashin_et_al_22,Postnikov06}. Despite extensive work in classifying linear preservers for such matrix classes, many questions remain open. For example, the classification of linear operators on the space of $n\times n$ complex matrices that map the set of positive semidefinite matrices into itself remains unresolved.
	
	The goal of this survey is to present the recent progress on linear preserver problems for certain positivity classes of matrices. The problems under consideration in this survey can be formulated as follows: given a finite-dimensional vector space $V$ and a subset $\mathcal{S}\subseteq V$, characterize all linear operators
	\begin{align*}
		\mathcal{L}:V\to V \quad \text{such that} \quad \mathcal{L}(\mathcal{S}) = \mathcal{S} \quad \text{or} \quad \mathcal{L}(\mathcal{S}) \subseteq \mathcal{S}.
	\end{align*}
	If $\mathcal{L}$ satisfies the former condition, it is called an \textit{onto $\mathcal{S}$-preserver}; if it satisfies the latter, it is called an \textit{into $\mathcal{S}$-preserver}. In our discussion, $V$ will denote either the space of $m\times n$ real/complex matrices or the real space of symmetric/Hermitian matrices, and $\mathcal{S}$ will represent a positivity class.
	
	Before proceeding further, we record the following observation: if the set $\mathcal{S}$ contains a basis of $V$, then every onto $\mathcal{S}$-preserver must be an invertible linear map. Consequently, the inverse map is also an $\mathcal{S}$-preserver, which makes their classification comparatively easier. In contrast, the characterization of into linear preservers is considerably more challenging. As a result, researchers often impose additional constraints on linear preservers, typically assuming them to be invertible or to possess a special structural form. For instance, Hershkowitz--Johnson~\cite{HJ86} classified invertible into preservers of P-matrices and several closely related matrix classes, while Furtado--Johnson--Zhang~\cite{Furtado_et_al_21} and Dorsey--Gannon--Jacobson--Johnson--Turnansky~\cite{Dorsey_et_al_16} characterized linear transformations of standard form that preserve the copositive and semipositive matrix classes, respectively, in the into sense. In such cases, the classification problem reduces to determining the structural conditions that these linear maps must satisfy.
	
	\section{Linear maps preserving positivity classes of matrices}
	
	This section is devoted to linear preservers associated with various positivity classes of matrices. For each such class and its corresponding body of work, we present a dedicated subsection in which we highlight key results, identify gaps in the existing literature, and formulate open questions that naturally arise from the discussion. To set the stage, we start by recalling some basic definitions, notations, and a preliminary result.
	
	\begin{defn}
		Let $m,n\geq1$ be positive integers throughout.
		\begin{itemize}
			\item[(1)] We denote by $[n]$ the set of the first $n$ positive integers.
			
			\item[(2)] Let $\mathcal{S}_n$ denote the subspace of all $n\times n$ real symmetric matrices.
			
			\item[(3)] We define $\mathcal{D}(n)$ to be the set of all $n\times n$ diagonal matrices, and $\mathcal{D}_{\scriptscriptstyle{>0}}(n)$ the set of all $n\times n$ diagonal matrices with positive diagonal entries. Similarly, one can define $\mathcal{D}_{\scriptscriptstyle{\geq0}}(n)$.
			
			\item[(4)] Let $\mathbb{R}^{m\times n}_{\scriptscriptstyle\geq0}$ denote the set of all $m\times n$ matrices with nonnegative entries. If $A\in\mathbb{R}^{m\times n}_{\scriptscriptstyle\geq0}$, we refer to $A$ as a \textit{nonnegative matrix}. 
			We define $\mathbb{R}^n_{\scriptscriptstyle>0}$ as the set of all $n$-dimensional vectors with positive entries. One can similarly define $\mathbb{R}^n_{\scriptscriptstyle\geq0}$. 
			
			
			\item[(5)] A \textit{signature matrix} is a diagonal matrix with diagonal entries in $\{\pm1\}$.
			
			\item[(6)] A matrix with exactly one non-zero entry in each row and column, and zeros elsewhere, is called a \textit{monomial matrix}; equivalently, it is the product of a diagonal matrix with non-zero diagonal entries and a permutation matrix.
			
			\item[(7)] Let $A=(a_{ij})$ and $B=(b_{ij})$ be two matrices of the same size. The \textit{Hadamard (or entrywise) product} of $A$ and $B$ is defined by
			\begin{align*}
				A\circ B:=(a_{ij}b_{ij}).
			\end{align*}
			
			\item[(8)] Let $A,B\in\mathbb{C}^{n\times n}$. We say that $A$ is \textit{congruent} to $B$ if there exists an invertible matrix $S\in\mathbb{C}^{n\times n}$ such that $A=SBS^T$. If instead $A = SBS^*$, then $A$ is said to be \textit{$^*congruent$} to $B$. In addition, if the matrix $S$ is monomial, these relations are referred to as \textit{monomial congruence} and \textit{monomial $^*congruence$}, respectively.
			
			
			
			\item[(9)] A linear map $\mathcal{L}:\mathbb{R}^{m\times n}\to\mathbb{R}^{m\times n}$ is said to be of \textit{standard form} if it admits a representation of the form
			\begin{align*}
				\mathcal{L}(A) = XAY \quad \text{or} \quad \mathcal{L}(A) = XA^TY,
			\end{align*}
			where $X$ and $Y$ are fixed matrices of appropriate sizes. Such a transformation is bijective if and only if both $X$ and $Y$ are invertible.
		\end{itemize}
	\end{defn}
	
	We now recall a lemma that plays a key role in classifying linear preservers of subsets of a finite-dimensional vector space via the linear preservers of their closures.
	
	\begin{lemma}\cite[Lemma 1]{BHJ85}\label{lemma_preserver}
		Let $V$ be a finite-dimensional vector space and let $\mathcal{S}\subseteq V$. Suppose that $P(\mathcal{S})$ denotes the set of all onto $\mathcal{S}$-preservers. Then $P(\mathcal{S})\subseteq P(\overline{\mathcal{S}})$.
	\end{lemma}
	
	We next turn our attention to several positivity classes of matrices and their linear preservers.
	
	\subsection{Positive semidefinite matrices}
	
	Let $A\in\mathbb{C}^{n\times n}$ be such that $A=A^*$. Then $A$ is called \textit{positive definite} (PD) if $\mathbf{x}^{*}A\mathbf{x}>0$ for all $\mathbf{0}\neq\mathbf{x}\in\mathbb{C}^n$. If $\mathbf{x}^{*}A\mathbf{x}\geq0$ for all $\mathbf{x}\in\mathbb{C}^n$, then $A$ is said to be \textit{positive semidefinite} (PSD). The set of PSD matrices forms a convex cone, as it is closed under addition and under multiplication by positive scalars. However, the usual matrix product of two PSD/PD matrices need not be PSD/PD (or even symmetric), whereas their entrywise (or Hadamard) product is always PSD/PD (see~\cite[Theorem 7.5.3]{Johnson_Horn12}). The classes of positive semidefinite and postive definite matrices have applications in different areas such as optimization and mathematical programming, signal processing, quantum information theory, statistics, and in all downstream applications via correlation and covariance matrices. Here are a few examples of these matrices.
	\begin{example}
		\begin{itemize}
			\item[(1)] For every $A\in\mathbb{C}^{m\times n}$, the matrices $AA^*$ and $A^*A$ are positive semidefinite. Moreover, if $\rk A = m$, then $AA^*$ is positive definite, and if $\rk A = n$, then $A^*A$ is positive definite.
			
			
			\item[(2)] Let $(V,\langle \cdot,\cdot\rangle)$ be an inner product space and let $\mathbf{v}^{(1)},\ldots,\mathbf{v}^{(n)}\in V$. Then the matrix
			\begin{align*}
				G = [\langle \mathbf{v}^{(j)},\mathbf{v}^{(i)}\rangle]_{i,j=1}^n\in\mathbb{C}^{n\times n},
			\end{align*}
			known as the \textit{Gram matrix}, is positive semidefinite. The matrix $G$ is positive definite if and only if $\mathbf{v}^{(1)},\ldots,\mathbf{v}^{(n)}$ are linearly independent.
			
			\item[(3)] Let $A=(a_{ij})\in\mathbb{C}^{n\times n}$ be Hermitian. If $A$ is diagonally dominant with nonnegative diagonal entries, i.e.
			\begin{align*}
				a_{ii} \geq \displaystyle\sum_{\substack{j=1 \\ j\neq i}}^{n} |a_{ij}|, \quad i=1,\ldots,n,
			\end{align*}
			then $A$ is positive semidefinite. 
%
			If $A$ is strictly diagonally dominant with positive diagonal entries, then $A$ is positive definite.
			
			\item[(4)] The $n\times n$ Cauchy matrix $A=(a_{ij}) = \left(\frac{1}{i+j}\right)_{i,j=1}^n$ is positive definite.
		\end{itemize}
	\end{example}
	Positive semidefinite matrices admit several equivalent characterizations; some of which are summarized in the result below.
	\begin{theorem}\cite[Chapter 7]{Johnson_Horn12}
		For a Hermitian matrix $A\in\mathbb{C}^{n\times n}$, the following statements are equivalent.
		\begin{itemize}
			\item[(1)] $A$ is positive semidefinite.
			
			\item[(2)] All eigenvalues of $A$ are nonnegative.
			
			\item[(3)] All principal minors of $A$ are nonnegative.
			
			\item[(4)] There exists a Hermitian matrix $C\in\mathbb{C}^{n\times n}$ such that $A=C^2$.
			
			\item[(5)] There exists a matrix $B\in\mathbb{C}^{n\times k}$  for some $k \geq 1$ such that $A=BB^*$.
			
			\item[(6)] There exists a lower triangular matrix $L\in\mathbb{C}^{n\times n}$ such that $A=LL^*$.
			
			\item[(7)] There exists a Euclidean vector space $V$ of dimension $k$ and vectors $\mathbf{v}^{(1)},\ldots,\mathbf{v}^{(n)}\in V$ such that $A = [\langle \mathbf{v}^{(i)},\mathbf{v}^{(j)}\rangle]_{i,j=1}^n$.
			
			\item[(8)] There exists a sequence of positive definite matrices $\{A_k\}_{k=1}^{\infty}$ such that $\displaystyle\lim_{k\to\infty}A_k = A$.
		\end{itemize}
	\end{theorem}
	
	The classification of linear transformations preserving positive semidefiniteness constitutes one of the earliest results in preservers literature for matrix positivity classes. The first complete result in this direction dates back to 1965, when Schneider~\cite{Schneider65} obtained a characterization of onto linear preservers of the cone of positive semidefinite matrices. We state this classical theorem below.
	
	\begin{theorem}\cite[Theorem 2]{Schneider65}\label{Theorem_PSD}
		Let $\mathcal{H}_n$ denote the real vector space of $n\times n$ Hermitian matrices. Suppose $\mathcal{L}:\mathcal{H}_n\to\mathcal{H}_n$ is a linear transformation such that $\mathcal{L}(PSD)=PSD$. Then
		\begin{align*}
			\mathcal{L}(A)=RAR^* \quad \text{or} \quad \mathcal{L}(A)=RA^TR^*,
		\end{align*}
		for some invertible matrix $R\in\mathbb{C}^{n\times n}$.
	\end{theorem}
	
	Observe that Theorem~\ref{Theorem_PSD} also applies to the set of positive definite matrices; we include a short proof for completeness.
	
	\begin{cor}\label{Cor_PD}
		All linear operators on the real vector space of $n\times n$ Hermitian matrices that map the set of positive definite matrices onto itself are of the form
		\begin{align}\label{Map_PD}
			A\mapsto RAR^* \quad \text{or} \quad A\mapsto RA^TR^*,
		\end{align}
		for some invertible matrix $R\in\mathbb{C}^{n\times n}$.
	\end{cor}
	\begin{proof}
		Clearly, both transformations in~\eqref{Map_PD} map the set of PD matrices onto itself. We now show the forward direction. Since the set of $n\times n$ positive definite matrices is dense in the set of $n\times n$ positive semidefinite matrices, it follows from Lemma~\ref{lemma_preserver} and Theorem~\ref{Theorem_PSD} that the onto linear preservers of positive definite matrices are among the maps in~\eqref{Map_PD}. This completes the proof.
	\end{proof}
	In contrast, the classification of into preservers of PSD matrices is considerably more subtle, and no complete description is currently known. Notice that while every $^*\mathrm{congruence}$ transformation maps the cone of positive semidefinite matrices into itself, there exist many additional into linear preservers of the PSD cone that are not onto preservers. For instance, the maps (i) $A\mapsto RAR^*$, for any given matrix $R\in\mathbb{C}^{n\times n}$, and (ii) $A\mapsto T\circ A$, for any fixed positive semidefinite matrix $T\in\mathbb{C}^{n\times n}$, are into preservers of the set of $n\times n$ PSD matrices.
	
	\subsection{Completely positive rank preservers}\label{SubSec_CP}
	
	A matrix $A\in\mathbb{R}^{n\times n}$ is called \textit{completely positive} (CP) if $A=BB^T$ for some matrix $B\in\mathbb{R}^{n\times k}_{\scriptscriptstyle\geq0}$. Writing $B$ in column form as $B:=[\mathbf{b}^{(1)},\ldots,\mathbf{b}^{(k)}]$, where each $\mathbf{b}^{(i)}\in\mathbb{R}^n_{\scriptscriptstyle\geq0}$, yields the factorization
	\begin{align}\label{CP_Factorization}
		A = \displaystyle\sum_{i=1}^{k}\mathbf{b}^{(i)}\mathbf{b}^{(i)^T}.
	\end{align}
	Note that these matrices form a special subclass of positive semidefinite matrices; specifically, they are entrywise nonnegative positive semidefinite matrices. Such matrices are called \textit{doubly nonnegative}. The study of CP matrices has its origins in inequality theory and quadratic forms~\cite{Diananda62,Hall_et_al_63}. For a detailed treatment of these matrices, we refer to the book by Shaked-Monderer--Berman~\cite{Monderer-Berman}.
	
	From the graph associated with a symmetric matrix, one can determine when a doubly nonnegative matrix is completely positive. This result is due to Berman--Grone~\cite{Berman_Grone}. To continue our discussion, we need the following definition.
	\begin{defn}
		For any $A=(a_{ij})\in\mathcal{S}_n$, the associated simple graph $G(A)$ has vertex and edge sets given by $	V(G(A))=\{1,\ldots,n\}$ and $E(G(A))=\{\{i,j\} \mid a_{ij}\neq0, ~ 1\leq i<j\leq n\}$, respectively.
			
	\end{defn}
	
	In 1988, Berman--Grone~\cite{Berman_Grone} obtained the following result.
	
	\begin{theorem}\cite[Theorem 3.1]{Berman_Grone}\label{CP_bipartite}
		If $A\in\mathcal{S}_n$ is doubly nonnegative and $G(A)$ is bipartite, then $A$ is completely positive.
	\end{theorem}
	
	In a similar vein, the cp-rank of a completely positive matrix $A$ can also be determined from its graph $G(A)$. The \textit{cp-rank} of a matrix $A\in\mathrm{CP}$ is defined as the least positive integer $k$ for which a factorization of the form~\eqref{CP_Factorization} exists. We denote the cp-rank of $A$ by $\cpr A$. In general, the cp-rank of a matrix may differ from its usual rank. Indeed, it follows immediately from the factorization~\eqref{CP_Factorization} that the rank of any completely positive matrix is bounded above by its cp-rank. Determining the cp-rank of a given completely positive matrix is, in general, not straightforward. However, if the graph associated with a completely positive matrix is triangle free, then its cp-rank can be determined immediately. The following result, due to Drew--Johnson--Loewy~\cite{Drew_et_al} in 1994, provides an explicit formula for the cp-rank in this setting.
	\begin{theorem}\cite[Theorems 5 and 6]{Drew_et_al}\label{Lemma_Drew}
		Let $G(A)$ be the connected graph corresponding to the $n\times n$ completely positive matrix $A$. If $G(A)$ is triangle free, then
		\begin{align*}
			\cpr A = \max\{n,|E(G(A))|\}.
		\end{align*}
	\end{theorem}
	
	We now proceed to discuss the problem of classifying linear maps that preserve the cp-rank of completely positive matrices. A linear transformation $\mathcal{L}:\mathcal{S}_m\to\mathcal{S}_n$ is called a \textit{cp-rank preserver} if, for every completely positive matrix $A\in\mathcal{S}_m$, the image $\mathcal{L}(A)\in\mathcal{S}_n$ is also completely positive and has the same cp-rank as $A$. Such linear maps not only preserve completely positive matrices but also preserve their cp-rank. In~\cite{Beasley_Song_et_al_16}, the authors claimed that if $m\leq n$, then a linear map $\mathcal{L}:\mathcal{S}_m\to\mathcal{S}_n$ is a cp-rank preserver if and only if
	\begin{align*}
		\mathcal{L}(A)=S^TAS,
	\end{align*}
	for some matrix $S\in\mathbb{R}_{\scriptscriptstyle\geq0}^{m\times n}$ of full row rank. This statement was later shown to be false in general by Shaked-Monderer~\cite{Monderer22}. She first proved that the statement is indeed valid for the cases $m=2,3$.
	
	\begin{theorem}\cite[Theorem 3.5]{Monderer22}\label{Monderer}
		Let $n\geq m$, where $m=2,3$. A linear transformation $\mathcal{L}:\mathcal{S}_m \to\mathcal{S}_n$ is a cp-rank preserver if and only if
		\begin{align*}
			\mathcal{L}(A)=S^TAS,
		\end{align*}
		for some $S\in\mathbb{R}^{m\times n}_{\scriptscriptstyle\geq0}$ of full row rank.
	\end{theorem}
	
	She further presented counterexamples showing that the characterization given in Theorem~\ref{Monderer} fails for $m\geq4$.
	\begin{example}\cite[Example 3.6]{Monderer22}
		Consider the matrices
		\begin{align*}
			S = \begin{bmatrix}
				1 & 0 & 0 & 0 \\
				1 & 1 & 0 & 0 \\
				1 & 1 & 1 & 0 \\
				1 & 1 & 1 & 1
			\end{bmatrix} \quad \text{and} \quad 	A = \begin{bmatrix}
				2 & 1 & 0 & 1 \\
				1 & 2 & 1 & 0 \\
				0 & 1 & 2 & 1 \\
				1 & 0 & 1 & 2
			\end{bmatrix}.
		\end{align*}
		Clearly, $S$ is nonnegative and nonsingular. One can also verify that $A$ is doubly nonnegative. Since $G(A)$ is bipartite, Theorem~\ref{CP_bipartite} implies that $A$ is completely positive. By Theorem~\ref{Lemma_Drew}, $\cpr A =4$. A direct computation gives
		\begin{align*}
			S^TAS = \begin{bmatrix}
				16 & 12 & 8 & 4 \\
				12 & 10 & 7 & 3 \\
				8 & 7 & 6 & 3 \\
				4 & 3 & 3 & 2
			\end{bmatrix} = \begin{bmatrix}
				4 & 0 & 0 \\
				3 & 1 & 0 \\
				2 & 1 & 1 \\
				1 & 0 & 1
			\end{bmatrix}\begin{bmatrix}
				4 & 0 & 0 \\
				3 & 1 & 0 \\
				2 & 1 & 1 \\
				1 & 0 & 1
			\end{bmatrix}^T.
		\end{align*}
		This factorization shows that $\cpr(S^TAS)\leq3$, but $\cpr A=4$. Thus, $S$ does not preserve cp-rank.
	\end{example}
	Shaked-Monderer conjectured that when $m=n\geq4$, a linear operator on $\mathcal{S}_n$ preserves cp-rank if and only if it is a monomial congruence transformation. She settled this conjecture for $4\leq n\leq6$.
	
	\begin{theorem}\cite[Theorem 4.5]{Monderer22}\label{Monderer_4n6}
		For $4\leq n\leq 6$, a linear map $\mathcal{L}:\mathcal{S}_n\to\mathcal{S}_n$ is a cp-rank preserver if and only if
		\begin{align*}
			\mathcal{L}(A) = Q^TAQ,
		\end{align*}
		where $Q\in\mathbb{R}^{n\times n}_{\scriptscriptstyle\geq0}$ is a monomial matrix.
	\end{theorem}
	
	The conjecture was subsequently established in full generality by Shitov~\cite{Shitov23} in 2023.
	
	\begin{theorem}\cite{Shitov23}\label{Shitov_CP}
		Let $n\geq 6$ and let $\mathcal{L}:\mathcal{S}_n\to\mathcal{S}_n$ be a linear map. Then $\mathcal{L}$ is a cp-rank preserver if and only if it has the form
		\begin{align*}
			\mathcal{L}(A) = Q^TAQ,
		\end{align*}
		for some monomial matrix $Q\in\mathbb{R}^{n\times n}_{\scriptscriptstyle\geq0}$.
	\end{theorem}
	
	Thus, when $n\geq m$ with $m\in\{2,3\}$, as well as when $m=n\geq4$, Theorem~\ref{Monderer} and Theorems~\ref{Monderer_4n6},~\ref{Shitov_CP} together provide a complete classification of all linear maps $\mathcal{L}:\mathcal{S}_m\to\mathcal{S}_n$ that preserve cp-rank. However, when $n>m\geq4$, the problem of characterizing all cp-rank preserving linear maps from $\mathcal{S}_m$ to $\mathcal{S}_n$ remains open.

	\subsection{Copositive matrices}
	
	A matrix $A\in\mathcal{S}_n$ is \textit{copositive} ($\mathrm{COP}_n$) if $\mathbf{x}^TA\mathbf{x}\geq0$ for all $\mathbf{x}\in\mathbb{R}^n_{\scriptscriptstyle\geq0}$. If $\mathbf{x}^TA\mathbf{x}>0$ for all $\mathbf{0}\neq\mathbf{x}\in\mathbb{R}^n_{\scriptscriptstyle\geq0}$, then $A$ is called \textit{strictly copositive} (SCOP$_n$). These classes of matrices were first defined by Motzkin~\cite{Motzkin52} in 1952. Copositive matrices arise in a variety of areas, including dynamical systems, control theory, and optimization. For a comprehensive treatment of these matrix classes, we refer to the book~\cite{Monderer-Berman}. For applications, see also the survey~\cite{Dur}.
	
	Observe that every positive semidefinite matrix and every nonnegative symmetric matrix is trivially copositive. Moreover, since the set of copositive matrices forms a cone, we immediately obtain
	\begin{align*}
		\mathrm{PSD} + (\mathcal{S}_n\cap\mathbb{R}^{n\times n}_{\scriptscriptstyle\geq0}) \subseteq \mathrm{COP_n}.
	\end{align*}
	It is noteworthy that for $n\leq4$, equality holds in the above relation. This was shown by Diananda~\cite{Diananda62} in 1962. However, for $n\geq5$, the inclusion is strict as illustrated by the following example due to Horn.
	
	\begin{example}\cite{Hall_et_al_63} Consider the Horn matrix
		\begin{align*}
			H = \begin{bmatrix}
				\phantom{-}1 & -1 & \phantom{-}1 & \phantom{-}1 & -1 \\
				-1 & \phantom{-}1 & -1 & \phantom{-}1 & \phantom{-}1 \\
				\phantom{-}1 & -1 & \phantom{-}1 & -1 & \phantom{-}1 \\
				\phantom{-}1 & \phantom{-}1 & -1 & \phantom{-}1 & -1 \\
				-1 & \phantom{-}1 & \phantom{-}1 & -1 & \phantom{-}1
			\end{bmatrix}.
		\end{align*}
		To show that $H$ is copositive, let $\mathbf{x}=(x_1,\ldots,x_5)^T\in\mathbb{R}^5_{\scriptscriptstyle\geq0}$. Then $\mathbf{x}^TH\mathbf{x}$ can be written in the following two ways:
		\begin{align}
			\mathbf{x}^TH\mathbf{x} &= (x_1 - x_2 + x_3 + x_4 - x_5)^2 + 4x_2x_4 + 4x_3(x_5-x_4) \label{Horn1}\\
			&= (x_1 - x_2 + x_3 - x_4 + x_5)^2 + 4x_2x_5 + 4x_1(x_4-x_5). \label{Horn2}
		\end{align}
		Thus, if $x_5\geq x_4$, then \eqref{Horn1} implies $\mathbf{x}^TH\mathbf{x}\geq0$. If $x_5<x_4$, then \eqref{Horn2} implies $\mathbf{x}^TH\mathbf{x}\geq0$. Hence, $H$ is copositive. However, $H$ is neither positive semidefinite nor nonnegative. Moreover, one can verify that $H$ is extremal in the cone of copositive matrices. Consequently, $H$ cannot be written as the sum of a positive semidefinite matrix and a nonnegative matrix.
		
		Let $\mathbf{0}_{n\times n}$ denote the $n\times n$ zero matrix. For $n>5$, the $n\times n$ block-diagonal matrix $H\oplus \mathbf{0}_{(n-5)\times(n-5)}$ provides a suitable counterexample.
	\end{example}
	Similar to positive semidefinite matrices, copositive matrices can also be characterized in terms of their spectral properties.
	
	\begin{theorem}\cite[Theorem 2]{Kaplan2000}
		Let $A\in\mathcal{S}_n$. Then $A$ is copositive if and only if no principal submatrix of $A$ has an eigenvector $\mathbf{v}\in\mathbb{R}^n_{\scriptscriptstyle>0}$ corresponding to a negative eigenvalue $\lambda<0$.
	\end{theorem}
	
	In 2013, Gowda--Sznajder--Tao~\cite{Gowda_Sznajder_Tao13} provided a characterization of the automorphisms of the copositive cone, that is, the bijective linear operators on $\mathcal{S}_n$ that map $\mathrm{COP}_n$ onto itself. More precisely, given a cone $\mathcal{C}\subseteq\mathbb{R}^n$, they introduced the associated copositive cone
	\begin{align}\label{Cone_CO}
		\mathrm{COP}_n(\mathcal{C}):=\{A\in\mathcal{S}_n \mid \mathbf{x}^TA\mathbf{x}\geq0 ~ \forall ~ \mathbf{x}\in\mathcal{C}\},
	\end{align}
	and established the following result.
	
	\begin{theorem}\cite[Corollary 12]{Gowda_Sznajder_Tao13} \label{Theorem_Gowda}
		Let $\mathcal{C}$ be a closed pointed cone in $\mathbb{R}^n$ with non-empty interior such that $\mathcal{C}\setminus\{\mathbf{0}\}$ is connected. Then every bijective linear map $\mathcal{L}:\mathcal{S}_n\to\mathcal{S}_n$ satisfying $\mathcal{L}\left(COP_n(\mathcal{C})\right) = COP_n(\mathcal{C})$ is of the form
		\begin{align*}
			\mathcal{L}(A) = QAQ^T,
		\end{align*}
		for some invertible $Q\in\mathbb{R}^{n\times n}$ such that $Q\mathcal{C}=\mathcal{C}$. 
	\end{theorem}
	By choosing $\mathcal{C}=\mathbb{R}^n_{\scriptscriptstyle\geq0}$, Theorem~\ref{Theorem_Gowda} yields a complete characterization of onto copositivity preservers -- a result independently obtained by Shitov~\cite[Theorem 1]{Shitov21} in 2021, without bijective assumption. For completeness, we now derive this classification using Theorem~\ref{Theorem_Gowda}.
	
	\begin{cor}\label{Cor_COP_onto}
		A linear operator $\mathcal{L}:\mathcal{S}_n\to\mathcal{S}_n$ maps COP$_n$ onto itself if and only if
		\begin{align*}
			\mathcal{L}(A)=QAQ^T,
		\end{align*}
		for some monomial matrix $Q\in\mathbb{R}^{n\times n}_{\scriptscriptstyle\geq0}$.
	\end{cor}
	\begin{proof}
		The converse is immediate. To prove the forward direction, let $\mathcal{C}=\mathbb{R}^n_{\scriptscriptstyle\geq0}$. Then $\mathcal{C}$ satisfies all the assumptions of Theorem~\ref{Theorem_Gowda}. Moreover, if $Q\in\mathbb{R}^{n\times n}$ is invertible and satisfies $Q\mathcal{C}=\mathcal{C}$, then $Q$ must be a nonnegative monomial matrix. 
		
		Next, observe that $\mathcal{L}$ is invertible: since COP$_n$ contains a basis of $\mathcal{S}_n$, and $\mathcal{L}$ maps COP$_n$ onto itself, we have
		\begin{align*}
			\mathcal{L}(\mathcal{S}_n) = \mathcal{L}\left(\mathrm{span}(\mathrm{COP}_n)\right) = \mathrm{span}\left(\mathcal{L}(\mathrm{COP}_n)\right) = \mathrm{span}(\mathrm{COP}_n) = \mathcal{S}_n.
		\end{align*}
		Thus, $\mathcal{L}$ is onto. Since $\mathcal{S}_n$ is finite-dimensional, $\mathcal{L}$ is therefore bijective. The result now follows from Theorem~\ref{Theorem_Gowda}.
	\end{proof}
	
	As a consequence of Corollary~\ref{Cor_COP_onto}, we obtain the linear preservers of strictly copositive matrices. We state this result below. Its proof is verbatim to that of Corollary~\ref{Cor_PD}.
	\begin{cor}\label{Cor_SCOP}
		A linear transformation $\mathcal{L}:\mathcal{S}_n\to\mathcal{S}_n$ that maps the set of strictly copositive matrices onto itself is of the form
		\begin{align*}
			\mathcal{L}(A)=QAQ^T,
		\end{align*}
		where $Q\in\mathbb{R}^{n\times n}_{\scriptscriptstyle\geq0}$ is a monomial matrix.
	\end{cor}
	
	Corollaries~\ref{Cor_COP_onto} and \ref{Cor_SCOP} establish that the onto linear preservers of copositivity are precisely the congruences by fixed nonnegative monomial matrices. On the other hand, the characterization of into linear preservers of copositivity is much more challenging. For instance, in~\cite{Pokora}, a conjecture of N.~Johnston was stated asserting that every into copositivity preserver must be of the form
	\begin{align}\label{Into_CO}
		A \mapsto \displaystyle\sum_i X_iAX_i^T,
	\end{align}
	where $X_i\in\mathbb{R}^{n\times n}_{\scriptscriptstyle\geq0}$. This map is evidently an into preserver of copositivity. However, not all into preservers admit such a representation, even in dimension $n=2$, as shown in the example below.
	
	\begin{example}\cite{Furtado_et_al_21}
		Consider the linear operator on $\mathcal{S}_2$ defined by
		\begin{align}\label{CO_Example}
			\begin{bmatrix}
				a & c \\
				c & b
			\end{bmatrix} \mapsto \begin{bmatrix}
				a & a + b + 2c \\
				a + b + 2c & b
			\end{bmatrix}.
		\end{align}
		If the preimage matrix is copositive, then $a+b+2c\geq0$, since this quantity equals the value of the quadratic form of the preimage evaluated at $\begin{pmatrix}
			1 \\
			1
		\end{pmatrix}$. Consequently, the image matrix has nonnegative entries and is therefore copositive. Hence, the operator~\eqref{CO_Example} preserves copositivity.
		
		Now observe that any map of the form~\eqref{Into_CO} is also a preserver of positive semidefiniteness. Consider the positive semidefinite matrix $\begin{bmatrix}
		10 & -1 \\
		-1 & 10
		\end{bmatrix}$. Its image under the map~\eqref{CO_Example} is $\begin{bmatrix}
		10 & 18 \\
		18 & 10
		\end{bmatrix}$, which is not positive semidefinite. This shows that the operator~\eqref{CO_Example} cannot be expressed in the form~\eqref{Into_CO}, despite being a copositivity preserver.
	\end{example}
	
	This leads Furtado--Johnson--Zhang~\cite{Furtado_et_al_21} to classify several special classes of into copositivity preservers. They first characterized all such maps in standard form. Note that if $\mathcal{L}$ is a linear operator on $\mathcal{S}_n$ of standard form, $\mathcal{L}(A)=XAY$, then $Y=X^T$.
	
	\begin{theorem}\cite[Theorem 2.2]{Furtado_et_al_21}\label{Theorem_COP_into_standard}
		Let $\mathcal{L}:\mathcal{S}_n\to\mathcal{S}_n$ be a linear map of the standard form,
		\begin{align*}
			\mathcal{L}(A) = SAS^T.
		\end{align*}
		Then $\mathcal{L}$ is an into preserver of $COP_n$ if and only if $S\in\mathbb{R}^{n\times n}_{\scriptscriptstyle\geq0}$.
	\end{theorem}
	Next, they considered linear maps of the form $A\mapsto H\circ A$, for a fixed matrix $H$. If $H\in\mathbb{R}^{n\times n}$ is a completely positive matrix, that is,
	\begin{align*}
		H = \displaystyle\sum_{i=1}^{k}\mathbf{h}^{(i)} \mathbf{h}^{(i)^T}, \quad \text{where} ~ \mathbf{h}^{(i)}\in\mathbb{R}^n_{\scriptscriptstyle\geq0} ~ \forall ~ i\in[k],
	\end{align*}
	then for every $A\in\mathrm{COP}_n$ and every $\mathbf{x}\in\mathbb{R}^n_{\scriptscriptstyle\geq0}$, we obtain
	\begin{align*}
		 \mathbf{x}^T(H\circ A)\mathbf{x} = \mathbf{x}^T\left(\displaystyle\sum_{i=1}^{k}\mathbf{h}^{(i)}\mathbf{h}^{(i)^T}\circ A\right)\mathbf{x} = \displaystyle\sum_{i=1}^{k}(\mathbf{x}\circ \mathbf{h}^{(i)})^TA(\mathbf{x}\circ \mathbf{h}^{(i)})\geq0.
	\end{align*}
	Thus, the matrix $H\circ A$ is copositive for all $A\in\mathrm{COP}_n$. The following result shows that every copositivity preserver of the form $A\mapsto H\circ A$, requires $H$ to be completely positive.
	
	\begin{theorem}\cite[Theorem 2.4]{Furtado_et_al_21}\label{Theorem_COP_into_Hadamard}
		Let $\mathcal{L}:\mathcal{S}_n\to\mathcal{S}_n$ be defined by
		\begin{align*}
			\mathcal{L}(A) = H\circ A.
		\end{align*}
		Then $\mathcal{L}$ is an into preserver of COP$_n$ if and only if $H\in\mathbb{R}^{n\times n}$ is a completely positive matrix.
	\end{theorem}
	
	It is worth noting that, for each $i\in[k]$, if we define the matrix $D_i:=\diag(\mathbf{h}^{(i)})\in\mathcal{D}_{\scriptscriptstyle\geq0}(n)$, then
	\begin{align*}
		H\circ A = \displaystyle\sum_{i=1}^{k}D_i^TAD_i.
	\end{align*}
	Consequently, such into preservers can be expressed as finite sums of into maps of standard form. 
%
	Finally, Furtado--Johnson--Zhang~\cite{Furtado_et_al_21} demonstrated a construction of into copositivity preservers as follows: every linear operator $\mathcal{L}:\mathcal{S}_n\to\mathcal{S}_n$ can be expressed as
	\begin{align*}
		\mathcal{L}(A)=(l_{ij}(A)),
	\end{align*}
	where each $l_{ij}$ is a linear functional on the vector space $\mathcal{S}_n$. Symmetry requires that $l_{ij}=l_{ji}$ for all $i,j\in[n]$. For each pair $(i,j)$, choose vectors $\mathbf{x}_{ij}^{(k)}=\mathbf{x}_{ji}^{(k)}$ with $\mathbf{x}_{ij}^{(k)}\in\mathbb{R}^n_{\scriptscriptstyle\geq0}$, and define
	\begin{align*}
		l_{ij}(A) := \displaystyle\sum_k \mathbf{x}_{ij}^{(k)^T}A\mathbf{x}_{ij}^{(k)}.
	\end{align*}
	This construction ensures that $\mathcal{L}(A)$ is symmetric and entrywise nonnegative for all $A\in\mathrm{COP}_n$, and hence copositive. Note that an analogous construction yields into preservers of positive semidefinite matrices as well. 
	
	We conclude this subsection by noting that the onto preservers of the cone COP$_n$ (see Corollary~\ref{Cor_COP_onto}), together with the into preservers of the special forms described in Theorems~\ref{Theorem_COP_into_standard}~and~\ref{Theorem_COP_into_Hadamard}, are, respectively, onto and into preservers of the cone of PSD matrices. However, there exist positive semidefinite preservers which fail to preserve copositivity, as demonstrated by the classical example below.
		
	\begin{example}
		The Choi map, introduced by Choi in~\cite{Choi}, is a linear map $\phi:\mathcal{S}_n\to\mathcal{S}_n$ defined by
		\begin{align*}
			\phi\left(\begin{bmatrix}
				a_{11} & a_{12} & a_{13} \\
				a_{21} & a_{22} & a_{23} \\
				a_{31} & a_{32} & a_{33}
			\end{bmatrix}\right) = \begin{bmatrix}
				2a_{11} + 2a_{22} & -a_{12} & -a_{13} \\
				-a_{21} & 2a_{22} + 2a_{33} & -a_{23} \\
				-a_{31} & -a_{32} & 2a_{33} + 2a_{11}
			\end{bmatrix}.
		\end{align*}
		One can verify that $\phi$ preserves all $3\times3$ PSD matrices. However, $\phi$ fails to preserve copositivity; in particular, it does not preserve copositive matrices with zero diagonal entries and positive off-diagonal entries.
	\end{example}
	

	\subsection{Sign regular matrices}
	
	Let $A\in\mathbb{R}^{m\times n}$. We say that $A$ is \textit{strictly sign regular of order $k$} (SSR$_k$) if there exists a sequence of signs $\epsilon_i\in\{\pm1\}$ such that every $i\times i$ minor of $A$ has sign $\epsilon_i$ for all $i\in[k]$. If zero minors are also permitted, then $A$ is called \textit{sign regular of order $k$} (SR$_k$). When $k=\min\{m,n\}$, then the matrix $A$ is correspondingly said to be \textit{strictly sign regular} (SSR) and \textit{sign regular} (SR). Moreover, for an SSR (respectively, SR) matrix $A$, if $\epsilon_i=1$ for all $i\in[\min\{m,n\}]$, then $A$ is called \textit{totally positive} (TP) (respectively, \textit{totally nonnegative} (TN)). These matrix classes have been extensively studied across diverse areas of mathematics, e.g., analysis, algebra, combinatorics, matrix theory, probability and statistics, etc. For a detailed discussion on these matrix classes, see the classical 1941 book by Gantmacher--Krein~\cite{GK50} (for SSR and SR matrices), and the monograph by Pinkus~\cite{pinkus} (for TP and TN matrices). See also Karlin's book~\cite{Karlin64} for a general framework based on kernels.
	
%
%
The classes of SSR, SR, TP, and TN matrices are closed under usual matrix multiplication. This is an immediate consequence of the Cauchy--Binet formula. We now present some examples of TP/TN matrices.
    \begin{example}
    	\begin{itemize}
%
%
    		\item[(1)] The \textit{generalized Vandermonde matrix} $V=(x_j^{\alpha_k})_{j,k=1}^{m,n}$ is totally positive for real numbers $0<x_1<\cdots<x_m$ and $\alpha_1<\cdots <\alpha_n$.
    		
    		\item[(2)] For $\sigma>0$ and real numbers $x_1<\cdots <x_n$, $y_1<\cdots<y_n$, the matrix $(e^{-\sigma(x_i-y_j)^2})^n_{i,j=1}$ is totally positive.
    		
    		
    		\item[(3)] The \textit{Toeplitz cosine matrix} $C(\theta):= \cos((j-k)\theta)^n_{j,k=1}$ is totally nonnegative for $\theta\in \left[0,{\textstyle \frac{\pi}{2(n-1)}}\right]$, with $n\geq2$.
    		
    		\item[(4)] The \textit{Cauchy matrix} $A = \left(\frac{1}{x_i+y_j}\right)_{i,j=1}^n$ is totally positive for real numbers $0<x_1<\cdots<x_n$ and $0<y_1<\cdots<y_n$.
    	\end{itemize}
    \end{example}
    
    There are several other well-known classes of examples of TP and TN matrices, including Green matrices, Jacobi matrices, Hankel matrices; see~\cite[Chapter 4]{pinkus} for more examples. Recently, Choudhury--Yadav~\cite{CY-SSR_Construction24} provided an algorithm to construct explicit examples of SSR matrices of any given size and sign pattern.

   We next discuss onto linear preservers of the sets of SSR and SR matrices -- both for fixed sign patterns and for all sign patterns -- which were recently characterized by Choudhury--Yadav~\cite{CY-LPP24}. To present their results, we first introduce the relevant definitions and notations.
	
	\begin{defn}
		\begin{itemize}
			\item[(1)] The \textit{sign pattern} of an SSR/SR matrix $A\in\mathbb{R}^{m\times n}$ is the ordered tuple $\epsilon=(\epsilon_1,\ldots,\epsilon_{\min\{m,n\}})$, where $\epsilon_i\in\{\pm1\}$ denotes the sign of all $i\times i$ minors of $A$.
			
			\item[(2)] We denote by SSR$(\epsilon)$ and SR$(\epsilon)$ the classes of $m\times n$ SSR and SR matrices, respectively, whose sign pattern is given by $\epsilon$. Similarly, one can define SSR$_k(\epsilon)$ and SR$_k(\epsilon)$ for $k\in[\min\{m,n\}]$.
			
			\item[(3)] An \textit{exchange matrix} is an anti-diagonal matrix whose all anti-diagonal elements are $1$. We denote an $n\times n$ exchange matrix by
			\begin{align*}
				P_n:=\begin{bmatrix}
					0 & \cdots & 1 \\
					\vdots & \iddots & \vdots \\
					1 & \cdots & 0 \\
				\end{bmatrix} \in \{ 0, 1 \}^{n\times n}.
			\end{align*}
		\end{itemize}
	\end{defn}
	
	 We also recall here a classical 1941 result of Gantmacher--Krein~\cite{GK50}.
	
	\begin{theorem}\cite{GK50}\label{G-K}
		Let $m,n\geq k\geq1$ be integers. Given a sign pattern $\epsilon=(\epsilon_1,\ldots,\epsilon_k)$, the set of $m\times n$ SSR$_k(\epsilon)$ matrices is dense in the set of $m\times n$ SR$_k(\epsilon)$ matrices.
	\end{theorem} 
	Lemma~\ref{lemma_preserver}, together with Theorem~\ref{G-K}, will be useful in classifying linear preservers of SSR and TP matrices. We now state a result that characterizes the onto linear preservers of the class of (strictly) sign regular matrices allowing all sign patterns.
	
	\begin{theorem}\cite[Theorem A]{CY-LPP24} \label{Theorem_SR}
		Suppose $m,n \geq 2$ are integers with $\max\{m,n\}\geq3$, and let SR and SR$_2$ denote the classes of all $m\times n$ sign regular and sign regular matrices of order 2, respectively. For a linear operator $\mathcal{L}:\mathbb{R}^{m\times n}\to\mathbb{R}^{m\times n}$, the following statements are equivalent.
		\begin{itemize}
			\item[(1)] $\mathcal{L}(SR)=SR$.
			\item[(2)] $\mathcal{L}(SR_2)=SR_2$.
			\item[(3)] $\mathcal{L}$ is a composition of one or more of the following types of transformations:
			\begin{itemize}
				\item[(a)] $A\mapsto FAE$, where $F\in\mathcal{D}_{\scriptscriptstyle{>0}}(m)$ and $E\in\mathcal{D}_{\scriptscriptstyle{>0}}(n)$;
				\item[(b)] $A\mapsto -A$;
				\item[(c)] $A\mapsto P_mA$, in which $P_m$ is an exchange matrix;
				\item[(d)] $A\mapsto AP_n$;  and
				\item[(e)] $A\mapsto A^T$, provided $m=n$.
			\end{itemize}
		\end{itemize}
		Moreover, the theorem is also true if SR is replaced by SSR in parts (1) and (2).
	\end{theorem}
	
	Note that when $m=n$, the transformation (3)(d) is not required, given (c) and (e). Moreover, it is noteworthy that the classification of linear preservers of $\mathrm{SR}$ reduces to the study of linear preservers of $\mathrm{SR}_2$. For dimensions not covered in Theorem~\ref{Theorem_SR}, Choudhury--Yadav~\cite{CY-LPP24} made the following observations.
	
	\begin{rem}
		If $m \neq n$ and $\min\{m,n\}=1$ in Theorem~\ref{Theorem_SR}, the problem reduces to characterizing linear $\mathrm{SR}_1$-preservers. In this case, Theorem~\ref{Theorem_SR} remains valid, but the second statement now requires $\mathcal{L} \in P (\mathrm{SR}_1)$ in place of  $\mathcal{L} \in P(\mathrm{SR}_2)$. Moreover, in the third statement, $P_m, P_n$ can be any permutation matrices instead of exchange matrices. If $m=n=2$, there is a unique $2\times 2$ minor, and the set $\mathcal{S}$ consists of SR matrices of all sign patterns in Theorem~\ref{Theorem_SR}. Consequently, the problem reduces to classifying linear $\mathrm{SR}_1$-preservers. The authors addresses this latter case in the following theorem.
	\end{rem}
	
	\begin{theorem}\cite[Theorem B]{CY-LPP24}\label{thrmSR_1}
		Given a linear map $\mathcal{L}:\mathbb{R}^{2\times2}\to\mathbb{R}^{2\times 2}$, the following are equivalent.
		\begin{itemize}
			\item[(1)] $\mathcal{L}(SR)=SR$.
			\item[(2)] $\mathcal{L}(SR_1)=SR_1$.
			\item[(3)] $\mathcal{L}$ is a composition of one or more of the following types of transformations:
			\begin{itemize}
				\item[(a)] $A\mapsto H\circ A$, where $H\in\mathbb{R}^{2\times2}$ is an entrywise positive matrix;
				\item[(b)] $A\mapsto -A$;
				\item[(c)] $A\mapsto P_2A$, in which $P_2$ is an exchange matrix;
				\item[(d)] $A\mapsto A^T$; and
				\item[(e)] $\begin{bmatrix}
					a_{11} & a_{12} \\
					a_{21} & a_{22}
				\end{bmatrix}\mapsto\begin{bmatrix}
					a_{11} & a_{12} \\
					a_{22} & a_{21}
				\end{bmatrix}$.
			\end{itemize}
		\end{itemize}
		Moreover, the theorem is also true if SR$_1$ is replaced by SSR$_1$ in part (2).
	\end{theorem}
	
	Note that the transformation (e) in Theorem~\ref{thrmSR_1} need not preserve the invertibility of $2\times 2$ matrices. Thus, Choudhury--Yadav classified the onto linear preservers for the set of $2\times 2$ SSR matrices in the theorem below.
	
	\begin{theorem}\cite[Theorem 2.17]{CY-LPP24}\label{Theorem_SSR2}
		Given a linear map $\mathcal{L}:\mathbb{R}^{2\times2}\to\mathbb{R}^{2\times 2}$, the following are equivalent.
		\begin{itemize}
			\item[(1)] $\mathcal{L}(SSR) = SSR$.
			\item[(2)] $\mathcal{L}$ is a composition of one or more of the following types of transformations:
			\begin{itemize}
				\item[(a)] $A\mapsto FAE$, where $F,E\in\mathcal{D}_{\scriptscriptstyle{>0}}(2)$;
				\item[(b)] $A\mapsto -A$;
				\item[(c)] $A\mapsto P_2A$, in which $P_2$ is an exchange matrix; and
				\item[(d)] $A\mapsto A^T$.
			\end{itemize}
		\end{itemize}
	\end{theorem}
	
	Thus, the three theorems above yield a complete classification of onto linear preservers of the class of $m\times n$ SR/SSR matrices of all sign patterns for any given dimension. Next, observe that the map $A\mapsto -A$ in Theorems~\ref{Theorem_SR}, \ref{thrmSR_1}, and \ref{Theorem_SSR2} preserves the sign regularity of a matrix, although it does change its sign pattern. This observation led Choudhury--Yadav to pose the following question: for any given size $m\times n$ and sign pattern $\epsilon$, can one consider the set of all $m\times n$ SR$(\epsilon)$/SSR$(\epsilon)$ matrices and characterize their linear preservers? They affirmatively answered this question in the next theorem.
	
	\begin{theorem}\cite[Theorem C]{CY-LPP24}\label{Theorem_SR_Fixed}
		Suppose $m,n\geq2$ are integers and $\epsilon$ is a given sign pattern. Then, for a linear transformation $\mathcal{L}:\mathbb{R}^{m\times n}\to\mathbb{R}^{m\times n}$, the following statements are equivalent.
		\begin{itemize}
			\item[(1)] $\mathcal{L}(SR(\epsilon)) = SR(\epsilon)$.
			\item[(2)] $\mathcal{L}(SR_2(\epsilon)) = SR_2(\epsilon)$.
			\item[(3)] $\mathcal{L}$ is a composition of one or more of the following types of transformations:
			\begin{itemize}
				\item[(a)] $A\mapsto FAE$, where $F\in\mathcal{D}_{\scriptscriptstyle{>0}}(m)$ and $E\in\mathcal{D}_{\scriptscriptstyle{>0}}(n)$;
				\item[(b)] $A\mapsto P_mAP_n$, where $P_m, P_n$ are exchange matrices; and
				\item[(c)] $A\mapsto A^T$, provided $m=n$.
			\end{itemize}
		\end{itemize}
		Moreover, the theorem is also true if SR$(\epsilon)$ is replaced by SSR$(\epsilon)$.
	\end{theorem}
	This theorem reveals that it is again sufficient to study the linear preservers of $\mathrm{SR}_2(\epsilon)$ in order to characterize linear $\mathrm{SR}(\epsilon)$-preservers.
	\begin{rem}
		We finish this subsection with two remarks.
		\begin{itemize}
			\item[(1)] Choudhury--Yadav showed that the linear preservers of the classes of $m\times n$ $\mathrm{SR}$ and $\mathrm{SR}_2$ matrices (respectively, $\mathrm{SR}(\epsilon)$ and $\mathrm{SR}_2(\epsilon)$) given in Theorem~\ref{Theorem_SR} (respectively, Theorem~\ref{Theorem_SR_Fixed}) coincide with the linear preservers of the class of $m\times n$
			$\mathrm{SR}_k$ matrices (respectively, $\mathrm{SR}_k(\epsilon)$) for all $2\leq k\leq\mathrm{min}\{m,n\}$; see \cite[Remark 1.3]{CY-LPP24}.
			
			\item[(2)]  Observe that by taking $m=n$ and $\epsilon_i=1$ for each $i$, Theorem~\ref{Theorem_SR_Fixed} yields a classification of all onto linear preservers for the classes of square TP and TN matrices, which were characterized by Berman--Hershkowitz--Johnson in 1985; see~\cite[Theorem 3]{BHJ85}.
		\end{itemize}
	\end{rem}
	
	\subsection{M-matrices, Inverse M-matrices, and H-matrices}
	
	A matrix $A=(a_{ij})\in\mathbb{R}^{n\times n}$ is called a \textit{Z-matrix} if $a_{ij}\leq0$ for all $i\neq j$. Thus, we can always write such matrices as $A=s I_{n}-B$, where $I_n$ is the $n\times n$ identity matrix and $B\in\mathbb{R}^{n\times n}_{\scriptscriptstyle\geq0}$. If $s>\rho(B)$, with $\rho(B)$ denoting the spectral radius of $B$, then $A$ is called an \textit{M-matrix}. Thus, M-matrices are always invertible. If $s\geq\rho(B)$, then $A$ is called an \textit{M$_0$-matrix}. The letter M in M-matrices stands for Minkowski. This terminology appears to have been first introduced by Ostrowski~\cite{Ostrowski37,Ostrowski56} in connection with the work of Minkowski~\cite{Minkowski1900,Minkowski1907}, who proved that if a Z-matrix has all row sums positive, then its determinant is positive. The classes of M- and M$_0$-matrices play a central role in iterative methods in numerical analysis of dynamical systems, finite difference schemes for partial differential equations, linear complementarity problems, and the study of Markov processes.
	
	In their 1994 book~\cite{Berman_Plemmons}, Berman--Plemmons provided 50 equivalent conditions to a Z-matrix being an M-matrix. We state some of them below.
	
	\begin{theorem}\cite[Chapter 6]{Berman_Plemmons}
		Let $A\in\mathbb{R}^{n\times n}$ be a Z-matrix. Then each of the following conditions is equivalent to saying that $A$ is an M-matrix.
		\begin{itemize}
			\item[(1)] All real eigenvalues of every principal submatrix of $A$ are positive.
			\item[(2)] All principal minors of $A$ are positive.
			\item[(3)] All leading principal minors of $A$ are positive.
			\item[(4)] $A^{-1}$ exists and is nonnegative.
			\item[(5)] $A$ is monotone, i.e., if $A\mathbf{x}\in\mathbb{R}^{n}_{\scriptscriptstyle\geq0}$, then $\mathbf{x}\in\mathbb{R}^n_{\scriptscriptstyle\geq0}$ for all $\mathbf{x}\in\mathbb{R}^n$.
			\item[(6)] There exists $\mathbf{x}\in\mathbb{R}^n_{\scriptscriptstyle\geq0}$ such that $A\mathbf{x}\in\mathbb{R}_{\scriptscriptstyle>0}$.
		\end{itemize}
	\end{theorem}
	
	We now discuss an important class of matrices closely related to M-matrices, namely inverse M-matrices. A matrix $A\in\mathbb{R}^{n\times n}$ is called an \textit{inverse M-matrix} (IM) if $A^{-1}$ exists and $A^{-1}$ is an M-matrix. Note that an IM-matrix is necessarily a nonnegative matrix. This class of matrices is of independent interest, as IM-matrices arise in several applications, including taxonomy, the Ising model of ferromagnetism, and random energy models in statistical physics. For further details on inverse M-matrices, see the survey~\cite{Johnson_et_al_Inverse_M} and the references therein.
	
	In 1985, Berman--Hershkowitz--Johnson~\cite{BHJ85} characterized the onto linear preservers of the classes of M- and M$_0$-matrices. Later, in 1990, Tam--Liou~\cite{Tam_IM90} gave a full description of onto linear preservers of the class of IM-matrices and its closure. We state their results below.
	
	\begin{theorem}\cite[Theorem 2]{BHJ85}\cite[Theorem 4.1 ]{Tam_IM90}\label{Theorem_M}
		Let $\mathcal{L}:\mathbb{R}^{n\times n}\to\mathbb{R}^{n\times n}$ be a linear transformation, and let $\mathcal{S}\subseteq\mathbb{R}^{n\times n}$ be any one of the classes M, M$_0$, IM, $\overline{IM}$. Then the following statements are equivalent.
		\begin{itemize}
			\item[(1)] $\mathcal{L}$ maps $\mathcal{S}$ onto itself.
			
			\item[(2)] $\mathcal{L}$ is a composition of one or more of the following types of transformations:
			\begin{itemize}
				\item[(a)] $A\mapsto FAE$, where $F,E\in\mathcal{D}_{\scriptscriptstyle{>0}}(n)$;
				\item[(b)] $A\mapsto A^T$; and
				\item[(c)] $A\mapsto Q^TAQ$, where $Q\in\mathbb{R}^{n\times n}$ is a permutation matrix.
			\end{itemize}
		\end{itemize}
	\end{theorem}
	
	In the same work, using Theorem~\ref{Theorem_M} Berman--Hershkowitz--Johnson, classified the onto linear preservers of H- and H$_0$-matrices, respectively. We first recall their definitions.
	\begin{defn}
		For $A=(a_{ij})\in\mathbb{C}^{n\times n}$, the \textit{comparison matrix} $\mathcal{M}(A)$ is defined entrywise by
		\begin{align*}
			\mathcal{M}(A)_{ij} := \begin{cases}
				\phantom{-}|a_{ii}|, &\text{if} ~ i=j, \\
				-|a_{ij}|, &\text{if} ~ i\neq j.
			\end{cases}
		\end{align*}
		A matrix $A\in\mathbb{C}^{n\times n}$ is called an \textit{H-matrix} if $\mathcal{M}(A)$ is an M-matrix, that is,
		\begin{align*}
			H := \{A\in\mathbb{C}^{n\times n} \mid \mathcal{M}(A)\in \mathrm{M}\}.
		\end{align*}
		Similarly, $A\in\mathbb{C}^{n\times n}$ is called an \textit{H$_0$-matrix} if $\mathcal{M}(A)$ is an M$_0$-matrix.
	\end{defn}
	
	\begin{cor}\cite[Corollary 1]{BHJ85}
		Let $\mathcal{L}:\mathbb{C}^{n\times n}\to\mathbb{C}^{n\times n}$ be a linear transformation. Then the following statements are equivalent.
		\begin{itemize}
			\item[(1)] $\mathcal{L}$ maps the class of $n\times n$ H$_0$-matrices onto itself.
			
			\item[(2)] $\mathcal{L}$ is a composition of one or more of the following types of transformations:
			\begin{itemize}
				\item[(a)] $A\mapsto FAE$, where $F,E\in\mathcal{D}_{\scriptscriptstyle{>0}}(n)$;
				\item[(b)] $A\mapsto A^T$;
				\item[(c)] $A\mapsto Q^TAQ$, where $Q\in\mathbb{R}^{n\times n}$ is a permutation matrix; and
				\item[(d)] $A\mapsto K\circ A$, where $K=(k_{ij})$ satisfies $|k_{ij}|=1$ for all $i,j\in[n]$.
			\end{itemize}
		\end{itemize}
		Moreover, the theorem is also true if H$_0$ is replaced by H.
	\end{cor}
	
	We end this subsection by noting that the study of into linear preservers of matrix classes discussed in this subsection are unexplored.
	
	\subsection{P-matrices, $\mathcal{D}$-stable matrices, and $\mathcal{A}$-matrices}
	
	A matrix $A\in\mathbb{C}^{n\times n}$ is called a \textit{P-matrix} if all of its principal minors are positive. If all principal minors of $A$ are nonnegative, then $A$ is said to be a \textit{P$_0$-matrix}. The class of P-matrices include some important positivity classes such as positive definite matrices, TP matrices, the (inverse) M-matrices. The earliest systematic study of P-matrices is due to Fiedler--Pt\'ak~\cite{Fiedler-Ptak}. Since this foundational work, the class P and its subclasses have emerged as a rich area of research. P-matrices play a pivotal role in numerous applications, including the linear complementarity problem, global univalence of maps, and interval matrices. For a detailed account of P-matrices, see the book~\cite{Johnson_Smith_Tsatsomeros20}.
	
	We now present a classical characterization of real P-matrices due to Gale--Nikaido~\cite{Gale_Nikaido}.
	
	\begin{theorem}\cite[Theorem 2]{Gale_Nikaido}\label{P_SNP}
		Let $A\in\mathbb{R}^{n\times n}$. Then $A$ is a P-matrix if and only if, for all $\mathbf{0}\neq\mathbf{x}\in \mathbb{R}^n$, the vector $\mathbf{x}\circ (A\mathbf{x})$ has at least one positive entry.
	\end{theorem}
	
	We next discuss some more classes of matrices.
	\begin{defn}
		\begin{itemize}
			\item[(1)] A matrix $A\in\mathbb{C}^{n\times n}$ is called \textit{$\mathcal{D}$-stable} if all eigenvalues of $DA$ have positive real parts for every $D\in\mathcal{D}_{\scriptscriptstyle{>0}}(n)$. If the eigenvalues of $DA$ have nonnegative real parts for every $D\in\mathcal{D}_{\scriptscriptstyle{>0}}(n)$, then $A$ is said to be \textit{$\mathcal{D}$-semistable}. We denote by $\mathcal{D}$ and $\mathcal{D}_0$ the classes of $\mathcal{D}$-stable and $\mathcal{D}$-semistable matrices, respectively.
			
			\item[(2)] Let $\mathcal{A}$ denote the class of all matrices $A\in\mathbb{C}^{n\times n}$ for which there exists a matrix $D\in\mathcal{D}_{\scriptscriptstyle{>0}}(n)$ such that $AD+DA^*$ is positive definite. Similarly, $\mathcal{A}_0$ denotes the class of all matrices $A\in\mathbb{C}^{n\times n}$ for which there exists $D\in\mathcal{D}_{\scriptscriptstyle{>0}}(n)$ such that $AD+DA^*$ is positive semidefinite.
		\end{itemize}
	\end{defn}
	These matrix classes satisfy the following inclusion relations:
	\begin{align*}
		\mathcal{A}\subseteq \mathrm{P} \qquad \text{and} \qquad \mathcal{A}\subseteq\mathcal{D}.
	\end{align*}
	In 1985, Berman--Hershkowitz--Johnson~\cite{BHJ85} studied onto linear preservers of P-matrices, $\mathcal{D}$-stable matrices, and $\mathcal{A}$-matrices. To state their result, we first set some notation. We denote by P$_{k}$ and P$_{0k}$ the classes of $n\times n$ matrices whose principal minors of order at most $k$ ($k<n$) are positive and nonnegative, respectively. The class P$_0^+$ consists of all P$_0$-matrices with the additional requirement that there is at least one positive principal minor of each size.
	
	\begin{theorem}\cite[Theorem 1]{BHJ85}\label{Theorem_BHJ}
		Let $\mathcal{L}:\mathbb{R}^{n\times n}\to\mathbb{R}^{n\times n}$ be a linear transformation, and let $\mathcal{S}\subseteq\mathbb{R}^{n\times n}$ be any one of the classes P$_0$, P, P$_0^+$, P$_{03}$, P$_3$, $\mathcal{D}_0$, $\mathcal{D}$, $\mathcal{A}_0$, and $\mathcal{A}$. Then the following statements are equivalent.
		\begin{itemize}
			\item[(1)] $\mathcal{L}$ maps $\mathcal{S}$ onto itself.
			
			\item[(2)] $\mathcal{L}$ is a composition of one or more of the following types of transformations:
			\begin{itemize}
				\item[(a)] $A\mapsto FAE$, where $F,E\in\mathcal{D}_{\scriptscriptstyle{>0}}(n)$;
				\item[(b)] $A\mapsto A^T$;
				\item[(c)] $A\mapsto Q^TAQ$, where $Q\in\mathbb{R}^{n\times n}$ is a permutation matrix; and
				\item[(d)] $A\mapsto SAS$, where $S\in\mathcal{D}(n)$ is a signature matrix.
			\end{itemize}
		\end{itemize}
	\end{theorem}
	
	The above theorem provides a complete characterization of onto linear preservers of P-matrices and several classes closely related to it, as well as of $\mathcal{D}$-stable matrices and $\mathcal{A}$-matrices. In the same work, Berman--Hershkowitz--Johnson showed that linear maps on the space of $n\times n$ complex matrices that preserve the classes of P$_0$, P, P$_0^+$, P$_{03}$, P$_3$, $\mathcal{D}_0$, $\mathcal{D}$, $\mathcal{A}_0$, and $\mathcal{A}$ matrices can be obtained by arguments similar to those used in the proof of Theorem~\ref{Theorem_BHJ}, with minor modifications.
	
	We now shift our focus to into preservers. In 1986, Hershkowitz--Johnson~\cite{HJ86} examined into linear preservers of the class of P-matrices, under some additional constraints. We begin by discussing examples of linear maps that are into preservers of P-matrices. First, note that since into preservers need not be bijective, they allow additional maps that are excluded in the onto setting, as illustrated by the following example.
	
	\begin{example}\cite[Example 1]{HJ86}
		Let $A=(a_{ij})$, and define
		\begin{align*}
			\mathcal{L}(A) = \begin{bmatrix}
				a_{11} & l_{12} & \cdots & l_{1n} \\
				0      & a_{22} & \cdots & l_{2n} \\
				\vdots & \vdots & \ddots & \vdots \\
				0      &  0     & \cdots & a_{nn}  
			\end{bmatrix},
		\end{align*}
		where each $l_{ij}$ with $i<j$ is a linear functional on the space $\mathbb{C}^{n\times n}$. Clearly, $\mathcal{L}(\mathrm{P})\subseteq \mathrm{P}$, but $\mathcal{L}$ is singular.
	\end{example}
	Motivated by such examples, Hershkowitz--Johnson restricted their attention to invertible into preservers of P-matrices. Second, observe that the condition ``invertible and into'' is weaker than the condition ``onto''. As a result, there exist many linear maps that are invertible into preservers without being onto preservers, as demonstrated by the next example.
	
	\begin{example}\cite[Example 2]{HJ86}
		Let $A=(a_{ij})$, and define
		\begin{align*}
			\mathcal{L}(A) = \begin{bmatrix}
				2a_{11} & a_{12} & \cdots & a_{1n} \\
				a_{21}  & 2a_{22} & \cdots & a_{2n} \\
				\vdots & \vdots & \ddots & \vdots \\
				a_{n1} &  a_{n2} & \cdots & 2a_{nn}  
			\end{bmatrix}.
		\end{align*}
		Then $\mathcal{L}$ is invertible and satisfies $\mathcal{L}(\mathrm{P})\subseteq \mathrm{P}$, but $\mathcal{L}(\mathrm{P})\neq \mathrm{P}$.
	\end{example}
	We now turn to a discussion of the main results obtained by Hershkowitz--Johnson.
	
	\begin{theorem}\cite[Theorem 1]{HJ86}\label{Theorem_P_into}
		Let $n\geq2$, and let $\mathcal{L}:\mathbb{C}^{n\times n}\to\mathbb{C}^{n\times n}$ be a linear transformation such that
		\begin{align}\label{P_into_condition}
			N(\mathcal{L})\cap \{A:=(a_{ij}) \mid a_{ii} =0 ~ \forall ~ i\in[n]\} = \{\mathbf{0}_{n\times n}\},
		\end{align}
		where $N(\mathcal{L})$ denotes the null space of $\mathcal{L}$. Then the following statements are equivalent.
		\begin{itemize}
			\item[(1)] $\mathcal{L}$ maps the class of $n\times n$ P$_0$-matrices into themselves.
			\item[(2)] $\mathcal{L}$ is a composition of one or more of the following types of transformations:
			\begin{itemize}
				\item[(a)] Case 1: $n=2$.
				\begin{itemize}
					\item[(i)] $A\mapsto FAE$, where $F,E\in\mathcal{D}(2)$ with $FE\in\mathcal{D}_{\scriptscriptstyle{>0}}(2)$;
					\item[(ii)] $A\mapsto A^T$; and
					\item[(iii)] $\begin{bmatrix}
						a_{11} & a_{12} \\
						a_{21} & a_{22}
					\end{bmatrix} \mapsto \begin{bmatrix}
						k_1a_{11} + k_2a_{22} & a_{12} \\
						a_{21} & t_1a_{11} + t_2a_{22}
					\end{bmatrix}$, with $k_1, k_2, t_1, t_2 \geq0$ such that either
					\begin{align*}
						k_1t_2 + k_2t_1 \geq 1 \quad \text{or} \quad 1-2(k_1t_2+k_2t_1)+(k_1t_2-k_2t_1)^2\leq0.
					\end{align*}
				\end{itemize}
				
				\item[(b)] Case 2: $n\geq3$.
				\begin{itemize}
					\item[(i)] $A\mapsto FAE$, where $F,E\in\mathcal{D}(n)$ with $FE\in\mathcal{D}_{\scriptscriptstyle{>0}}(n)$;
					\item[(ii)] $A\mapsto A^T$;
					\item[(iii)] $A\mapsto Q^TAQ$, where $Q\in\mathbb{R}^{n\times n}$ is a permutation matrix; and
					\item[(iv)] $A\mapsto A+D$, where $D\in\mathcal{D}(n)$ whose diagonal entries are nonnegative linear combinations of the diagonal entries of $A$.
				\end{itemize}
			\end{itemize}
		\end{itemize}
		Moreover, the statement is also true for the case 1 if the class P$_0$ is replaced by any of: P, P$_0^+$, and for case 2 if the class P$_0$ is replaced by any of: P, P$_0^+$, P$_{0k}$, P$_k$, P$_{0k}^+$, for $k\geq4$.
	\end{theorem}
	
	Using Theorem~\ref{Theorem_P_into}, Hershkowitz--Johnson next classified all invertible into linear preservers of the class of P-matrices. Notice that if $\mathcal{L}$ is invertible, then condition~\eqref{P_into_condition} is automatically satisfied. Therefore, it only remains to identify the invertible linear preservers described in Theorem~\ref{Theorem_P_into}.
	
	\begin{cor}\cite{HJ86}\label{Cor_P_into}
		Let $n\geq2$, and let $\mathcal{L}:\mathbb{C}^{n\times n}\to\mathbb{C}^{n\times n}$ be an invertible linear transformation. Then the following statements are equivalent.
		\begin{itemize}
			\item[(1)] $\mathcal{L}$ maps the class of $n\times n$ P$_0$-matrices into themselves.
			\item[(2)] $\mathcal{L}$ is a composition of one or more of the following types of transformations:
			\begin{itemize}
				\item[(a)] Case 1: $n=2$.
				\begin{itemize}
					\item[(i)] $A\mapsto FAE$, where $F,E\in\mathcal{D}(2)$ with $FE\in\mathcal{D}_{\scriptscriptstyle{>0}}(2)$;
					\item[(ii)] $A\mapsto A^T$; and
					\item[(iii)] $\begin{bmatrix}
						a_{11} & a_{12} \\
						a_{21} & a_{22}
					\end{bmatrix} \mapsto \begin{bmatrix}
						k_1a_{11} + k_2a_{22} & a_{12} \\
						a_{21} & t_1a_{11} + t_2a_{22}
					\end{bmatrix}$, with $k_1, k_2, t_1, t_2 \geq0$ and $k_1t_2\neq k_2t_1$ such that either
					\begin{align*}
					     k_1t_2 + k_2t_1 \geq 1 \quad \text{or} \quad 1-2(k_1t_2+k_2t_1)+(k_1t_2-k_2t_1)^2\leq0.
					\end{align*}
				\end{itemize}
				
				\item[(b)] Case 2: $n\geq3$.
				\begin{itemize}
					\item[(i)] $A\mapsto FAE$, where $F,E\in\mathcal{D}(n)$ with $FE\in\mathcal{D}_{\scriptscriptstyle{>0}}(n)$;
					\item[(ii)] $A\mapsto A^T$;
					\item[(iii)] $A\mapsto Q^TAQ$, where $Q\in\mathbb{R}^{n\times n}$ is a permutation matrix; and
					\item[(iv)] $A\mapsto A+D$, with the map $\diag(A)\mapsto \diag(A+D)$ invertible. Here $D\in\mathcal{D}(n)$ whose diagonal entries are nonnegative linear combinations of the diagonal entries of $A$.
				\end{itemize}
			\end{itemize}
		\end{itemize}
		Moreover, the statement is also true for the case 1 if the class P$_0$ is replaced by any of: P, P$_0^+$, and for case 2 if the class P$_0$ is replaced by any of: P, P$_0^+$, P$_{0k}$, P$_k$, P$_{0k}^+$, for $k\geq4$.
	\end{cor}
	We conclude by noting some remarks of Hershkowitz--Johnson~\cite{HJ86} concerning into preservers of P-matrices over the real matrices. They pointed out that the proof of Theorem~\ref{Theorem_P_into} (and consequently Corollary~\ref{Cor_P_into}) does not hold for linear operators on the space of $n\times n$ real matrices when $n=2$ (see \cite[Example 3]{HJ86}). However, for the case $n\geq3$, they believed that Theorem~\ref{Theorem_P_into} (and hence Corollary~\ref{Cor_P_into}) may indeed remain valid in the real setting.
	
	\subsection{Semipositive matrices}
	
	A matrix $A\in\mathbb{R}^{m\times n}$ is called \textit{semipositive} (SP) if there exists a vector $\mathbf{x}\in\mathbb{R}^n_{\scriptscriptstyle\geq0}$ such that $A\mathbf{x}\in\mathbb{R}^{n}_{\scriptscriptstyle>0}$. A matrix $A$ is called \textit{minimally semipositive} (MSP) if it is semipositive and no submatrix obtained by deleting column of $A$ remains semipositive. In particular, an $m\times n$ MSP matrix must satisfy $m\geq n$. Further details on these classes can be found in the book~\cite{Johnson_Smith_Tsatsomeros20}.
	
	The class of semipositive matrices contains several notable positivity classes, including real P-matrices (see \cite[Theorem 4.3.6]{Johnson_Smith_Tsatsomeros20}), positive definite matrices, and entrywise positive matrices. Observe that not every SP matrix need to be a P-matrix. For example, any square matrix with all entries positive and a nonpositive determinant is an SP matrix, but not a P-matrix. However, real P-matrices admit a characterization via semipositivity.
	
%
%
%
	
	\begin{theorem}\cite[Theorem 1]{Nowaihi88}
		$A\in\mathbb{R}^{n\times n}$ is a $P$-matrix if and only if $SAS$ is semipositive for every signature matrix $S\in\mathcal{D}(n)$.
	\end{theorem}
	
	We next proceed to discuss the linear preservers of the classes of semipositive and minimally semipositive matrices. To present the main results, we first recall some definitions.
	\begin{defn}
		\begin{itemize}
			\item[(1)] A matrix $A\in\mathbb{R}^{n\times n}$ is called \textit{row positive} (RP) if $A\in\mathbb{R}^{n\times n}_{\scriptscriptstyle\geq0}$ and each row of $A$ contains at least one positive entry. 
			
			\item[(2)] A matrix $A\in\mathbb{R}^{m\times n}$ is called \textit{monotone} if $A\mathbf{x}\in\mathbb{R}^{m}_{\scriptscriptstyle\geq0}$, then $\mathbf{x}\in\mathbb{R}^n_{\scriptscriptstyle\geq0}$. The terminology ``monotone'' is motivated by the fact that $A\mathbf{x}\geq A\mathbf{y}$ (entrywise) implies $\mathbf{x}\geq\mathbf{y}$ (entrywise). 
		\end{itemize}
	\end{defn}
	
	Note that every nonnegative monomial matrix is row positive, but the converse need not hold. Finally, observe that if $A$ is a square monotone matrix, then $A^{-1}$ exists, and moreover, $A^{-1}\in\mathbb{R}^{n\times n}_{\scriptscriptstyle\geq0}$. In this case, the matrix $A$ is also referred to as \textit{inverse nonnegative} (IN).
	
	
	In 2016, Dorsey--Gannon--Jacobson--Johnson--Turnansky~\cite{Dorsey_et_al_16} examined onto and into linear preservers of semipositivity and minimally semipositivity. Their main results are summarized as follows.
	
	\begin{theorem}\cite[Theorems 2.4 and 2.11, Corollary 2.7]{Dorsey_et_al_16}\label{Dorsey}
		Let $\mathcal{L}:\mathbb{R}^{m\times n}\to\mathbb{R}^{m\times n}$ be a linear operator of the form $\mathcal{L}(A)=XAY$, where $X\in\mathbb{R}^{m\times m}$ and $Y\in\mathbb{R}^{n\times n}$ are fixed matrices. Then
		\begin{itemize}
			\item[(1)] $\mathcal{L}$ is an onto linear preserver of SP if and only if $X$ and $Y$ are nonnegative monomial matrices, or $-X$ and $-Y$ are nonnegative monomial matrices. Moreover, $\mathcal{L}$ is an into linear preserver of SP if and only if either $X$ is RP and $Y$ is IN, or $-X$ is RP and $-Y$ is IN.
			
			\item[(2)] For  $m>n$, $\mathcal{L}$ is an onto linear preserver of MSP if and only if $X$ and $Y$ are nonnegative monomial matrices, or $-X$ and $-Y$ are nonnegative monomial matrices. Moreover, for $m>n$, $\mathcal{L}$ is an into linear preserver of MSP if and only if $X$ is a nonnegative monomial matrix and $Y$ is IN, or $-X$ is a nonnegative monomial matrix and $-Y$ is IN.
		\end{itemize}
	\end{theorem}
	In~\cite{Choudhury_et_al_18}, Choudhury--Kannan--Sivakumar identified certain ambiguities in the preceding results. First, they provided a new proof of Theorem~\ref{Dorsey}(1) concerning into linear preservers of SP matrices (see \cite[Theorem 2.16]{Choudhury_et_al_18}). Second, they showed that the statement of Theorem~\ref{Dorsey}(2) does not hold in the setting of into linear preservers of MSP matrices, as shown by the example below.
	 
	\begin{example}\cite[Example 2.8]{Choudhury_et_al_18}
		Consider the following matrices
		\begin{align*}
			X=\begin{bmatrix}
				1 & 1 \\
				1 & 1
			\end{bmatrix} ~ \text{and} ~ Y=[1].
		\end{align*}
		Observe that $2\times1$ minimally semipositive matrices consist precisely of positive column vectors. Therefore, it follows that the map $\mathcal{L}(A)=XAY$ is an into preserver of minimally semipositive matrices. However, the matrix $X$ is not monomial.
	\end{example}
	
	The above example shows that the converse direction in Theorem~\ref{Dorsey}(2) fails for MSP matrices. In the same work, Dorsey et al. also classified general onto linear preservers of semipositivity, showing that they must be of standard form.
	
	\begin{theorem}\cite[Theorem 3.5]{Dorsey_et_al_16}\label{Theorem_SP_onto}
		A linear map $\mathcal{L}:\mathbb{R}^{m\times n}\to\mathbb{R}^{m\times n}$ is an onto preserver of semipositivity if and only if
		\begin{align*}
			\mathcal{L}(A)=XAY,
		\end{align*}
		for some monomial matrices $X\in\mathbb{R}^{m\times m}_{\scriptscriptstyle\geq0}$ and $Y\in\mathbb{R}^{n\times n}_{\scriptscriptstyle\geq0}$. 
	\end{theorem}
	
	In view of the above result, Dorsey et al. asked whether all into linear preservers of semipositivity must also be of the standard form. They showed that this is not the case by constructing the following counterexample.
	\begin{example}\cite{Dorsey_et_al_16}
		Let $m,n\geq2$ be integers. Consider the following linear map
		\begin{align*}
			\mathcal{L}\left(\begin{bmatrix}
				a_{11} & a_{12} & \cdots & a_{1n} \\
				a_{21} & a_{22} & \cdots & a_{2n} \\
				\vdots & \vdots & \cdots & \vdots \\
				a_{m1} & a_{m2} & \cdots & a_{mn}
			\end{bmatrix}\right) = \begin{bmatrix}
				a_{11} & a_{12} & \cdots & a_{1n} \\
				a_{11} & a_{12}-a_{11} & \cdots & a_{1n}-\displaystyle\sum_{i=1}^{n-1}a_{1i} \\
				\vdots & \vdots & \cdots & \vdots \\
				a_{11} & a_{12}-a_{11} & \cdots & a_{1n}-\displaystyle\sum_{i=1}^{n-1}a_{1i}
			\end{bmatrix}.
		\end{align*}
		Note that every semipositive matrix $A=(a_{ij})\in\mathbb{R}^{m\times n}$ must have at least one positive entry in each row, and in particular in the first row. Thus, we may assume that $a_{11},\ldots,a_{1k}\leq0$ for some $k\in[n-1]$ and $a_{1,k+1}>0$. Observe that the $(k+1)$th column of $\mathcal{L}(A)$ satisfies
		\begin{align*}
			(\mathcal{L}(A))_{i,k+1}=a_{1,k+1}-\displaystyle\sum_{i=1}^{k}a_{1i}>0 \quad \forall ~ i\in[m].
		\end{align*}
		Therefore, the vector $\mathcal{L}(A)\mathbf{e}^{(k+1)}$ has all entries positive, where $\mathbf{e}^{(k+1)}$ is the unit vector whose $(k+1)$th entry is $1$ and all remaining entries are $0$. This shows that $\mathcal{L}$ is indeed an into preserver of semipositivity. However, it cannot be expressed in the form $XAY$, since
		\begin{align*}
			\mathcal{L}\left(\begin{bmatrix}
				1 & 0 & \cdots & 0 \\
				0 & 0 & \cdots & 0 \\
				\vdots & \vdots & \ddots & \vdots \\
				0 & 0 & \cdots & 0
			\end{bmatrix}\right) = \begin{bmatrix}
				1 & \phantom{-}0 & \cdots & \phantom{-}0 \\
				1 & -1 & \cdots & -1 \\
				\vdots & \vdots & \ddots & \vdots \\
				1 & -1 & \cdots & -1
			\end{bmatrix},
		\end{align*}
		where the input matrix has rank 1 but the output matrix has rank 2. Since any map of the form $A\mapsto XAY$ satisfies $\rk (XAY)\leq \rk A$, this provides a contradiction.
	\end{example}
	
	On the other hand, Dorsey et al.~\cite{Dorsey_et_al_16} conjectured that, under the additional assumption that $\mathcal{L}$ is invertible, every into linear preserver of the class of $m\times n$ semipositive matrices must be of the form $XAY$ for $X$ row positive and invertible, and $Y$ inverse nonnegative. Subsequently, this conjecture was proved in 2021 by Jayaraman--Mer~\cite{Jayaraman_mer_21} for $m\geq n\geq2$.
	
	\begin{theorem}\cite[Theorem 3.3 and Remark 3.5, and Theorem 3.11]{Jayaraman_mer_21}
		Let $m\geq n\geq2$, and let $\mathcal{L}:\mathbb{R}^{m\times n}\to\mathbb{R}^{m\times n}$ be an invertible linear transformation. Then $\mathcal{L}$ is an into semipositivity preserver if and only if
		\begin{align*}
			\mathcal{L}(A) = XAY,
		\end{align*}
		where $X\in\mathbb{R}^{m\times m}$ is invertible and RP, and $Y\in\mathbb{R}^{n\times n}$ is IN.
	\end{theorem}
	
	We now note the following questions concerning minimally semipositive matrices:
	\begin{itemize}
		\item[(1)] What happens in Theorem~\ref{Dorsey}(2) when $m=n$, for both into and onto linear preservers of standard forms of MSP?
		\item[(2)] Can all onto linear preservers of MSP be characterized similarly to those of SP, as in Theorem~\ref{Theorem_SP_onto}?
	\end{itemize} 
	The first question was addressed by Choudhury--Kannan--Sivakumar~\cite{Choudhury_et_al_18} in 2018. Specifically, they proved the following result.
	\begin{theorem}\cite[Theorems 2.10, 2.11, 2.13, and 2.14]{Choudhury_et_al_18}
		Let $\mathcal{L}:\mathbb{R}^{n\times n}\to\mathbb{R}^{n\times n}$ be a linear transformation of the form $\mathcal{L}(A)=XAY$ or $\mathcal{L}(A)=XA^TY$, where $X$ and $Y$ are fixed real $n\times n$ matrices. Then
		\begin{itemize}
			\item[(1)] $\mathcal{L}$ is an onto linear preserver of MSP if and only if $X$ and $Y$ are nonnegative monomial matrices, or $-X$ and $-Y$ are nonnegative monomial matrices.
			
			\item[(2)] $\mathcal{L}$ is an into linear preserver of MSP if and only if $X$ and $Y$ are IN, or $-X$ and $-Y$ are IN.
		\end{itemize}
	\end{theorem}
	We conclude this subsection by noting that while significant progress has been made in classifying linear preservers of semipositivity, many questions remains unanswered, even for preservers of particular forms.
	
	In this survey, we have examined the linear preservers associated with several positivity classes of matrices. Although onto preservers have been determined in most cases, only certain special forms of into preservers are currently known. A complete classification of into preservers for many of these matrix classes remains an open problem. We end this survey by summarizing the results discussed thus far in Table~\ref{Table_LPP_Summary}. 

\renewcommand{\arraystretch}{1.40}
\small
	\begin{longtable}{|c|c|c|c|c|c|}
			\hline
			\parbox{0.5cm}{\textbf{S.} \\ \textbf{No.}} & $\mathbf{V}$ & $\mathbf{\mathcal{S}}$ & \parbox{1cm}{\textbf{Into/} \\ \textbf{onto}} & \textbf{Form of} $\mathbf{\mathcal{L}}$ & $\mathbf{\mathcal{S}}$\textbf{-preservers} \\ \hline\hline 
			
			(1) & \makecell{Real space of \\ $n\times n$ Hermitian \\ matrices} & \begin{tabular}[c]{@{}c@{}} PSD, \\ PD \end{tabular} & Onto & General & \makecell{$A\mapsto RAR^*$ or $A\mapsto RA^TR^*$, \\ $R\in\mathbb{C}^{n\times n}$ invertible matrix}  \\ \cline{6-6}
			&  & & & & Reference:\cite[Theorem 2]{Schneider65}  \\ \hline
			
			(2) & $\mathcal{S}_n$, $n=2,3$ & cp-rank & Onto & General & $A\mapsto S^TAS$, $S\in\mathbb{R}^{n\times n}_{\scriptscriptstyle\geq0}$ invertible matrix \\ \cline{6-6}
			&  & & & & Reference: \cite[Theorem 3.5]{Monderer22}  \\ \hline

			(3) & $\mathcal{S}_n$, $n\geq4$ & cp-rank & Onto & General & $A\mapsto Q^TAQ$, $Q\in\mathbb{R}^{n\times n}_{\scriptscriptstyle\geq0}$ monomial matrix \\ \cline{6-6}
			
			& & & & & References: \cite[Theorem 4.5]{Monderer22}, \cite{Shitov23} \\ \hline

			(4) & $\mathcal{S}_{n}$ & \begin{tabular}[c]{@{}c@{}} COP$_n$, \\ SCOP$_n$ \end{tabular} & Onto & General  & $A\mapsto Q^TAQ$, $Q\in\mathbb{R}^{n\times n}_{\scriptscriptstyle\geq0}$ monomial matrix \\ \cline{6-6}
			&  & & & & Reference: \cite[Theorem 1]{Shitov21}  \\ \hline

			\multirow{3}{*}{(5)} & \multirow{3}{*}{$\mathcal{S}_{n}$} & \multirow{3}{*}{COP$_n$} & \multirow{3}{*}{Into} & $A\mapsto XAY$ & $X=Y^T$, $X\in\mathbb{R}^{n\times n}_{\scriptscriptstyle\geq0}$ \\ \cline{6-6}
			
			& & & & & Reference: \cite[Theorem 2.2]{Furtado_et_al_21} \\ \cline{5-6}
			
			& & & & $A\mapsto H\circ A$ & $H\in\mathbb{R}^{n\times n}$ completely positive matrix\\ \cline{6-6}
			&  & & & & Reference: \cite[Theorem 2.4]{Furtado_et_al_21} \\ \hline
			
			(6) & 
			$\mathbb{R}^{2\times2}$ &
			\begin{tabular}[c]{@{}c@{}} SR, \\ SR$_1$ \end{tabular} & Onto & General &
			\begin{tabular}[c]{@{}l@{}}
				$A\mapsto H\circ A$, $H\in\mathbb{R}^{2\times2}$ is entrywise positive; \\
				$A\mapsto -A$; \\
				$A\mapsto P_2A$, $P_2\in\mathbb{R}^{2\times 2}$ an exchange matrix; \\
				$A\mapsto A^T$; \\
				$\begin{bsmallmatrix} a_{11} & a_{12} \\ a_{21} & a_{22} \end{bsmallmatrix} 
				\mapsto 
				\begin{bsmallmatrix} a_{11} & a_{12} \\ a_{22} & a_{21} \end{bsmallmatrix}$ \\
			\end{tabular}
			\\ \cline{6-6}
			 &  & & & & Reference: \cite[Theorem B]{CY-LPP24}  \\ \hline
			
			(7) &
			$\mathbb{R}^{2\times2}$ &
			\begin{tabular}[c]{@{}c@{}} SSR \end{tabular} & Onto & General &
			\begin{tabular}[c]{@{}l@{}}
				$A\mapsto FAE$, $F,E\in\mathcal{D}_{\scriptscriptstyle{>0}}(2)$; \\
				$A\mapsto -A$; \\
				$A\mapsto P_2A$, $P_2\in\mathbb{R}^{2\times2}$ an exchange matrix; \\
				$A\mapsto A^T$
			\end{tabular}
			\\ \cline{6-6}
			&  & & & & Reference: \cite[Theorem 2.17]{CY-LPP24} \\ \hline
			
			(8) & 
			\begin{tabular}[c]{@{}c@{}} 
				$\mathbb{R}^{m\times n}$ \\ 
				with \\ $m,n\geq2$ \& \\ $\max\{m,n\}\geq3$
			\end{tabular} &
			\begin{tabular}[c]{@{}c@{}} 
				SR, \\ SR$_k$, \\ SSR, \\ SSR$_k$; \\ 
				$2\leq k\leq$ \\ $\min\{m,n\}$
			\end{tabular} & Onto & General &
			\begin{tabular}[c]{@{}l@{}}
				$A\mapsto FAE$, $F\in\mathcal{D}_{\scriptscriptstyle{>0}}(m)$ and $E\in\mathcal{D}_{\scriptscriptstyle{>0}}(n)$; \\
				$A\mapsto -A$; \\
				$A\mapsto P_mA$, $P_m\in\mathbb{R}^{m\times m}$ an exchange matrix; \\
				$A\mapsto AP_n$, $P_n\in\mathbb{R}^{n\times n}$ an exchange matrix; \\
				$A\mapsto A^T$, provided $m=n$
			\end{tabular}
			\\ \cline{6-6}
			&  & & & & Reference: \cite[Theorem A]{CY-LPP24} \\ \hline
			
			(9) &
			\begin{tabular}[c]{@{}c@{}} 
				$\mathbb{R}^{m\times n}$ \\ 
				with \\ $m,n\geq2$
			\end{tabular} &
			\begin{tabular}[c]{@{}c@{}} 
				SR$(\epsilon)$, \\ SR$_k(\epsilon)$, \\ 
				SSR$(\epsilon)$, \\ SSR$_k(\epsilon)$; \\ 
				$2\leq k\leq$ \\ $\min\{m,n\}$
			\end{tabular} & Onto & General &
			\begin{tabular}[c]{@{}l@{}}
				$A\mapsto FAE$, $F\in\mathcal{D}_{\scriptscriptstyle{>0}}(m)$ and $E\in\mathcal{D}_{\scriptscriptstyle{>0}}(n)$; \\
				$A\mapsto P_mAP_n$, $P_m\in\mathbb{R}^{m\times m}$ and $P_n\in\mathbb{R}^{n\times n}$ \\ are exchange matrices; \\
				$A\mapsto A^T$, provided $m=n$
			\end{tabular}
			\\ \cline{6-6}
			&  & & & & Reference: \cite[Theorem C]{CY-LPP24} \\ \hline
			
			(10) &
			\begin{tabular}[c]{@{}c@{}} 
				$\mathbb{R}^{n\times n}$
			\end{tabular} &
			\begin{tabular}[c]{@{}c@{}} 
				M, \\ M$_0$, \\ IM, \\ $\overline{\mathrm{IM}}$
			\end{tabular} & Onto & General &
			\begin{tabular}[c]{@{}l@{}}
				$A\mapsto FAE$, $F,E\in\mathcal{D}_{\scriptscriptstyle{>0}}(n)$; \\ $A\mapsto A^T$; \\ $A\mapsto Q^TAQ$, $Q\in\mathbb{R}^{n\times n}$ permutation matrix
			\end{tabular}
			\\ \cline{6-6}
			&  & & & & References: \cite[Theorem 2]{BHJ85}, \cite[Theorem 4.1]{Tam_IM90} \\ \hline
			
			(11) &
			\begin{tabular}[c]{@{}c@{}} 
				$\mathbb{C}^{n\times n}$
			\end{tabular} &
			\begin{tabular}[c]{@{}c@{}} 
				H, \\ H$_0$
			\end{tabular} & Onto & General &
			\begin{tabular}[c]{@{}l@{}}
				$A\mapsto FAE$, $F,E\in\mathcal{D}_{\scriptscriptstyle{>0}}(n)$; \\
				$A\mapsto A^T$; \\ 
				$A\mapsto Q^TAQ$, $Q\in\mathbb{R}^{n\times n}$ permutation matrix; \\
				$A\mapsto K\circ A$, with $K=(k_{ij})$ is such that \\ $|k_{ij}|=1$ for all $i,j\in[n]$
			\end{tabular}
			\\ \cline{6-6}
			&  & & & & Reference: \cite[Corollary 1]{BHJ85} \\ \hline
			
			(12) &
			\begin{tabular}[c]{@{}c@{}} 
				$\mathbb{R}^{n\times n}$
			\end{tabular} &
			\begin{tabular}[c]{@{}c@{}} 
				P$_0$, \\ P, \\ P$_0^+$, \\ P$_{03}$, \\ P$_3$, \\ $\mathcal{D}_0$, \\ $\mathcal{D}$, \\ $\mathcal{A}_0$, \\ $\mathcal{A}$
			\end{tabular} & Onto & General &
			\begin{tabular}[c]{@{}l@{}}
				$A\mapsto FAE$, $F,E\in\mathcal{D}_{\scriptscriptstyle{>0}}(n)$; \\ $A\mapsto A^T$; \\ $A\mapsto Q^TAQ$, $Q\in\mathbb{R}^{n\times n}$ permutation matrix; \\ 
				$A\mapsto SAS$, $S\in\mathcal{D}(n)$ signature matrix
			\end{tabular}
			\\ \cline{6-6}
			&  & & & & Reference: \cite[Theorem 1]{BHJ85} \\ \hline
			
			(13) &
			\begin{tabular}[c]{@{}c@{}} 
				$\mathbb{C}^{n\times n}$
			\end{tabular} &
			\begin{tabular}[c]{@{}c@{}} 
				P$_0$, \\ P, \\ P$_0^+$, \\ P$_{03}$, \\ P$_3$, \\ $\mathcal{D}_0$, \\ $\mathcal{D}$, \\ $\mathcal{A}_0$, \\ $\mathcal{A}$
			\end{tabular} & Onto & General &
			\begin{tabular}[c]{@{}l@{}}
				$A\mapsto FAE$, $F,E\in\mathcal{D}(n)$ with $FE\in\mathcal{D}_{\scriptscriptstyle{>0}}(n)$; \\ $A\mapsto A^T$; \\ $A\mapsto Q^TAQ$, $Q\in\mathbb{R}^{n\times n}$ permutation matrix
			\end{tabular}
			\\ \cline{6-6}
			&  & & & & Reference: \cite{BHJ85} \\ \hline
			
%
			
			(14) &
			\begin{tabular}[c]{@{}c@{}} 
				$\mathbb{C}^{2\times 2}$
			\end{tabular} &
			\begin{tabular}[c]{@{}c@{}} 
				P$_0$, \\ P, \\ P$_0^+$
			\end{tabular} & Into & $\mathcal{L}$ invertible &
			\begin{tabular}[c]{@{}l@{}}
				$A\mapsto FAE$, $F,E\in\mathcal{D}(2)$ with $FE\in\mathcal{D}_{\scriptscriptstyle{>0}}(2)$; \\
				$A\mapsto A^T$; \\ 
				$\begin{bsmallmatrix}
					a_{11} & a_{12} \\
					a_{21} & a_{22}
				\end{bsmallmatrix}\mapsto \begin{bsmallmatrix}
					k_1a_{11} + k_2a_{22} & a_{12} \\
					a_{21} & t_1a_{11} + t_2a_{22}
				\end{bsmallmatrix}$, \\ with $k_1,k_2,t_1,t_2\geq0$ and $k_1t_2\neq k_2t_1$ \\ such that either
				$k_1t_2 + k_2t_1 \geq 1$ or \\ $1-2(k_1t_2+k_2t_1)+(k_1t_2-k_2t_1)^2\leq0$
			\end{tabular}
			\\ \cline{6-6}
			&  & & & & Reference: \cite{HJ86} \\ \hline
			
			(15) &
			\begin{tabular}[c]{@{}c@{}}
				$\mathbb{C}^{n\times n}, n\geq3$
			\end{tabular} &
			\begin{tabular}[c]{@{}c@{}}
				P$_0$, \\ P, \\ P$_0^+$, \\ P$_{0k}$, \\ P$_k$, \\ P$_{0k}^+$, \\ $k\geq4$
			\end{tabular} & Into & $\mathcal{L}$ invertible &
			\begin{tabular}[c]{@{}l@{}}
				$A\mapsto FAE$, $F,E\in\mathcal{D}(n)$ with $FE\in\mathcal{D}_{\scriptscriptstyle{>0}}(n)$; \\ $A\mapsto A^T$; \\ $A\mapsto Q^TAQ$, $Q\in\mathbb{R}^{n\times n}$ permutation matrix; \\ $A\mapsto A+D$, with $\diag(A)\mapsto\diag(A+D)$ \\ invertible. Here $D\in\mathcal{D}(n)$ whose diagonal \\ entries are nonnegative linear combinations \\ of the diagonal entries of A
			\end{tabular}
			\\ \cline{6-6}
			&  & & & & Reference: \cite{HJ86} \\ \hline
			
			\multirow{5}{*}{(16)} & \multirow{4}{*}{$\mathbb{R}^{m\times n}$} & \multirow{4}{*}{SP} & Onto & General & \begin{tabular}[c]{@{}l@{}}
				$A\mapsto XAY$, $X\in\mathbb{R}^{m\times m}_{\scriptscriptstyle\geq0}$, $Y\in\mathbb{R}^{n\times n}_{\scriptscriptstyle\geq0}$ are \\ monomial matrices
			\end{tabular} \\ \cline{6-6}
			
			& & & & & Reference: \cite[Theorem 3.5]{Dorsey_et_al_16} \\ \cline{4-6}
			
			& & & Into & $A\mapsto XAY$ & \begin{tabular}[c]{@{}l@{}}
				$X\in\mathbb{R}^{m\times m}$ is RP and $Y\in\mathbb{R}^{n\times n}$ is IN, or \\ $-X\in\mathbb{R}^{m\times m}$ is RP and $-Y\in\mathbb{R}^{n\times n}$ is IN
			\end{tabular} \\ \cline{6-6}
			
			& & & & & Reference: \cite[Theorem 2.4]{Dorsey_et_al_16}  \\ \cline{2-6}
			
			& \makecell{$\mathbb{R}^{m\times n}$, \\ $m\geq n\geq2$} & SP & Into & $\mathcal{L}$ invertible & \begin{tabular}[c]{@{}l@{}}
				$A\mapsto XAY$, $X\in\mathbb{R}^{m\times m}$ is invertible RP \\ and $Y\in\mathbb{R}^{n\times n}$ is IN
			\end{tabular} \\ \cline{6-6}
			&  & & & & Reference: \cite[Theorems 3.3, 3.11, Remark 3.5]{Jayaraman_mer_21} \\ \hline
			
			(17) & $\mathbb{R}^{m\times n}$, $m>n$ & MSP & Onto & $A\mapsto XAY$ & \begin{tabular}[c]{@{}l@{}}
					$X\in\mathbb{R}^{m\times m}_{\scriptscriptstyle\geq0}$, $Y\in\mathbb{R}^{n\times n}_{\scriptscriptstyle\geq0}$ are monomial matrices, or \\
					$-X\in\mathbb{R}^{m\times m}_{\scriptscriptstyle\geq0}$, $-Y\in\mathbb{R}^{n\times n}_{\scriptscriptstyle\geq0}$ are monomial matrices
			\end{tabular} \\ \cline{6-6} 
			
			& & & & & Reference: \cite[Theorem 2.11]{Dorsey_et_al_16} \\ \hline

			(18) & \multirow{2}{*}{$\mathbb{R}^{n\times n}$} & \multirow{2}{*}{MSP} & Onto & \makecell{$A\mapsto XAY$ \\ $A\mapsto XA^TY$} & \begin{tabular}[c]{@{}l@{}}
				$X\in\mathbb{R}^{n\times n}_{\scriptscriptstyle\geq0}$, $Y\in\mathbb{R}^{n\times n}_{\scriptscriptstyle\geq0}$ are monomial matrices, or \\
				$-X\in\mathbb{R}^{n\times n}_{\scriptscriptstyle\geq0}$, $-Y\in\mathbb{R}^{n\times n}_{\scriptscriptstyle\geq0}$ are monomial matrices
			\end{tabular} \\ \cline{6-6} 
			
			& & & & & Reference: \cite[Theorems 2.13, 2.14]{Choudhury_et_al_18}  \\ \cline{4-6}
			
			&  &  & Into & \makecell{$A\mapsto XAY$ \\ $A\mapsto XA^TY$} & \begin{tabular}[c]{@{}l@{}}
				$X,Y\in\mathbb{R}^{n\times n}$ are IN, or \\ $-X,-Y\in\mathbb{R}^{n\times n}$ are IN
			\end{tabular} \\ \cline{6-6}
			&  & & & & Reference: \cite[Theorems 2.10, 2.11]{Choudhury_et_al_18} \\ \hline
		\caption{Table summarizing linear $\mathcal{S}$-preservers defined on a vector space $V$.}
		\label{Table_LPP_Summary}
	\end{longtable}
    
	\subsubsection*{Acknowledgments}
	PNC was partially supported by ANRF Prime Minister Early Career Research Grant ANRF/ECRG/2024/002674/PMS  from ANRF, Govt.~of India and INSPIRE Faculty Fellowship Research Grant DST/INSPIRE/04/2021/002620 from DST, Govt.~of India. SY was supported by INSPIRE Faculty Fellowship Research Grant DST/INSPIRE/04/2021/002620 from DST, Govt.~of India (PI: Projesh Nath Choudhury).

%
%
%
%


\begin{thebibliography}{10}
		
%
%
%
%
		
        
        \bibitem{Nowaihi88}
        A.~al-Nowaihi.
        \newblock P-matrices: an equivalent characterisation.
        \newblock\href{https://doi.org/10.1016/0021-8693(88)90097-X}{\em J. Algebra}, 112:385--387, 1988.
        
        
        
        \bibitem{Berman_Grone}
        A.~Berman and R.~Grone.
        \newblock Bipartite completely positive matrices.
        \newblock\href{https://doi.org/10.1017/S0305004100064835}{\em Proc. Cambridge Philos. Soc.}, 103:269--276, 1988.
		
		
		\bibitem{BHJ85}
		A.~Berman, D.~Hershkowitz, and C.R.~Johnson.
		\newblock Linear transformations that preserve certain positivity classes of matrices.
		\newblock\href{https://doi.org/10.1016/0024-3795(85)90205-8}{\em Linear Algebra Appl.}, 68:9--29, 1985.
		
		
		
		
		
		
		\bibitem{Berman_Plemmons}
		A.~Berman and A.J.~Plemmons.
		\newblock {\em Nonnegative matrices in the mathematical sciences}.
		\newblock\href{https://epubs.siam.org/doi/10.1137/1.9781611971262}{SIAM}, 1994.
		
		

        \bibitem{Cao}
        C.~Cao and X.~Tang.
        \newblock Determinant preserving transformations on symmetric matrix spaces.
        \newblock\href{https://journals.uwyo.edu/index.php/ela/article/view/267}{\em Electron. J. Linear Algebra}, 11:205--211, 2004.
        
        
        \bibitem{Choi}
        M.-D.~Choi.
        \newblock Positive semidefinite biquadratic forms.
        \newblock\href{https://doi.org/10.1016/0024-3795(75)90058-0}{\em Linear Algebra Appl.}, 12:95--100, 1975.
		
		
		\bibitem{Choudhury_et_al_18}
		P.N.~Choudhury, M.R.~Kannan, and K.C.~Sivakumar.
		\newblock A note on linear preservers of semipositive and minimally semipositive matrices.
		\newblock\href{https://doi.org/10.13001/1081-3810.3864}{\em Electron. J. Linear Algebra}, 34:687--694, 2018.
		
		\bibitem{CY-SSR_Construction24}	
		P.N.~Choudhury and S.~Yadav.
		\newblock Constructing strictly sign regular matrices of all sizes and sign patterns.
		\newblock\href{https://londmathsoc.onlinelibrary.wiley.com/doi/10.1112/blms.70080}{\em Bull. London Math. Soc.}, 57:2077--2096, 2025.
		
		
		\bibitem{CY-LPP24}	
		P.N.~Choudhury and S.~Yadav.
		\newblock Sign regularity preserving linear operators.
		\newblock\href{https://doi.org/10.1112/blms.70209}{\em Bull. London Math. Soc.}, 58, article \#e70209 (24pp.), 2026.
		
		
		
		
		
		\bibitem{Diananda62}
		P.H.~Diananda.
		\newblock On nonnegative forms and in real variables some or all of which are nonnegative.
		\newblock\href{https://www.cambridge.org/core/journals/mathematical-proceedings-of-the-cambridge-philosophical-society/article/on-nonnegative-forms-in-real-variables-some-or-all-of-which-are-nonnegative/4500B3D1274524653579ACF856B29A47}{\em Proc. Cambridge Philos. Soc.}, 58:17--25, 1962.
		
		\bibitem{Dolinar}
		G.~Dolinar and P.~\v{S}emrl.
		\newblock Determinant preserving maps on matrix algebras.
		\newblock\href{https://www.sciencedirect.com/science/article/pii/S002437950100578X#:~:text=Determinant%20preserving%20maps%20on%20matrix%20algebras%E2%98%86&text=Let%20Mn%20be%20the,with%20det(MN)%3D1.}{\em Linear Algebra Appl.}, 348:189--192, 2002.
		
		
		\bibitem{Dorsey_et_al_16}
		J.~Dorsey, T.~Gannon, N.~Jacobson, C.R.~Johnson, and M.~Turnansky.
		\newblock Linear preservers of semi-positive matrices.
		\newblock\href{https://doi.org/10.1080/03081087.2015.1122723}{\em Linear Multilinear Algebra}, 64:1853--1862, 2016.
		
		\bibitem{Drew_et_al}
		J.H.~Drew, C.R.~ Johnson, and R.~Loewy.
		\newblock Completely positive matrices associated with $M$-matrices.
		\newblock\href{https://www.tandfonline.com/doi/abs/10.1080/03081089408818334}{\em Linear Multilinear Algebra}, 37:303--310, 1994.
		
		\bibitem{Dur}
		M.~D\"ur.
		\newblock Copositive programming -- a survey.
		\newblock\href{https://link.springer.com/chapter/10.1007/978-3-642-12598-0_1}{\em Recent Advances in Optimization and its Applications in Engineering}, Springer Berlin Heidelberg, 3--20, 2010.
		
		\bibitem{Eaten}
	    M.L.~Eaten.
		\newblock On linear transformations which preserve the determinant.
		\newblock\href{https://projecteuclid.org/journals/illinois-journal-of-mathematics/volume-13/issue-4/On-linear-transformations-which-preserve-the-determinant/10.1215/ijm/1256053433.full}{\em Illinois J. Math.}, 13:722--727, 1969.
		
		
		\bibitem{Fiedler-Ptak}
		M.~Fiedler and V.~Pt\'ak.
		\newblock Some generalizations of positive definiteness and monotonicity.
		\newblock\href{https://doi.org/10.1007/BF02166034}{\em Numer. Math.}, 9:163--172, 1966.
		
		\bibitem{FZ02}
		S.~Fomin and A.~Zelevinsky.
		\newblock Cluster algebras. {I}. {F}oundations.
		\newblock \href{http://dx.doi.org/10.1090/S0894-0347-01-00385-X}{\em J.\
			Amer.\ Math.\ Soc.}, 15(2):497--529, 2002.
		
		\bibitem{Fosner_et_al_survey}
		A.~Fo\v{s}ner, Z.~Huang, C.-K.~Li, and N.-S.~Sze. 
		\newblock Linear preservers and quantum information science.
		\newblock\href{https://doi.org/10.1080/03081087.2012.740029}{\em Linear Multilinear Algebra}, 61:1377--1390, 2013.
		
		
		
		
		\bibitem{Frobenius1897}
		G.~Frobenius.
		\newblock \"{U}ber die Darstellung der endlichen Gruppen durch lineare Substitutionen.
		\newblock {\em Sitzungsber. Preuss. Akad. Wiss. Berlin}, 994--1015, 1897.
		
		\bibitem{Fung96}
		H.-K.~Fung.
		\newblock Linear preservers of controllability and/or observability.
		\newblock\href{https://doi.org/10.1016/0024-3795(94)00364-5}{\em Linear Algebra Appl.}, 246:335--360, 1996.
		
		\bibitem{Furtado_et_al_21}
		S.~Furtado, C.R.~Johnson, and Y.L.~Zhang.
		\newblock Linear preservers of copositive matrices.
		\newblock\href{https://doi.org/10.1080/03081087.2019.1692775}{\em Linear Multilinear Algebra}, 69:1779--1788, 2021.
		
		\bibitem{Galashin_et_al_22}
		P.~Galashin, S.N.~Karp, and T.~Lam.
		\newblock Regularity theorem for totally nonnegative flag varieties. 
		\newblock\href{https://doi.org/10.1090/jams/983}{\em J. Amer. Math. Soc.}, 35(2):513--579, 2022.
		
		
		\bibitem{Gale_Nikaido}
		D.~Gale and H.~Nikaido.
		\newblock The Jacobian matrix and global univalence of mappings.
		\newblock\href{https://doi.org/10.1007/BF01360282}{\em Math. Ann.}, 159:81--93, 1965.
		
		\bibitem{GK37}
		F.R.~Gantmacher and M.G.~Krein.
		\newblock {Sur les matrices compl\`etement non n\'egatives et oscillatoires}.
		\newblock {\em Compositio Math.}, 4:445--476, 1937.
		
		\bibitem{GK50}
		F.R.~Gantmacher and M.G.~Krein.
		\newblock {\em Oscillyacionye matricy i yadra i malye kolebaniya mehani\v{c}eskih sistem}.
		\newblock Gosudarstv. Isdat. Tehn.-Teor. Lit., Moscow-Leningrad, 1941.
		
		
		
		
%
%
%
%

		
		
		\bibitem{Gowda_Sznajder_Tao13}
		M.S.~Gowda, R.~Sznajder, and J.~Tao.
		\newblock The automorphism group of a completely positive cone and its Lie algebra.
		\newblock\href{https://doi.org/10.1016/j.laa.2011.10.006}{\em Linear Algebra Appl.}, 438:3862--3871, 2013.
		
		\bibitem{Hahn}
		W.~Hahn.
		\newblock {\em Stability of motion.}
		\newblock\href{https://link.springer.com/book/10.1007/978-3-642-50085-5}{Springer Berlin, Heidelberg}, 1967.
		
		\bibitem{Hall_et_al_63}
		M.Jr.~Hall and M.~Newman.
		\newblock Copositive and completely positive quadratic forms.
		\newblock\href{https://www.cambridge.org/core/journals/mathematical-proceedings-of-the-cambridge-philosophical-society/article/copositive-and-completely-positive-quadratic-forms/2E22B1217CCADD28133ECC2383BDD65C}{\em Proc. Cambridge Philos. Soc.}, 59:329--339, 1963.
		
		
		
		
		
		\bibitem{HJ86}
		D.~Hershkowitz and C.R.~Johnson. 
		\newblock Linear transformations that map the $P$-matrices into themselves.
		\newblock\href{https://doi.org/10.1016/0024-3795(86)90113-8}{\em Linear Algebra Appl.}, 74:23--38, 1986.
		
		\bibitem{Johnson_Horn12}
		R.A.~Horn and C.R.~Johnson.
		\newblock {\em Matrix Analysis} (2nd ed.).
		\newblock\href{https://doi.org/10.1017/CBO9781139020411}{Cambridge: Cambridge University Press}, 2012.
		
%
%

		
		
		\bibitem{Jayaraman_mer_21}
		S.~Jayaraman and V.N.~Mer.
		\newblock On linear preservers of semipositive matrices.
		\newblock\href{https://doi.org/10.13001/ela.2021.5397}{\em Electron. J. Linear Algebra}, 37:88--112, 2021.
		
		
		\bibitem{Johnson_et_al_Inverse_M}
		C.R.~Johnson and R.L.~Smith.
		\newblock Inverse $M$-matrices II.
		\newblock\href{https://doi.org/10.1016/j.laa.2011.02.016}{\em Linear Algebra Appl.}, 435:953--983, 2011.
		
		\bibitem{Johnson_Smith_Tsatsomeros20}
		C.R.~Johnson, R.L.~Smith, and M.J.~Tsatsomeros.
		\newblock {\em Matrix Positivity}.
		\newblock\href{https://doi.org/10.1017/9781108778619}{Cambridge University Press}, 2020.
		
		\bibitem{Kaplan2000}
		W.~Kaplan.
		\newblock A test for copositive matrices.
		\newblock \href{https://www.sciencedirect.com/science/article/pii/S0024379500001385?via%3Dihub}{\em Linear Algebra Appl.}, 313:203--206, 2000.
		
		\bibitem{Karlin64}
		S.~Karlin.
		\newblock Total positivity, absorption probabilities and applications.
		\newblock \href{https://doi.org/10.2307/1993667}{\em Trans. Amer. Math. Soc.}, 111:33--107, 1964.
		
		\bibitem{K68}
		S.~Karlin.
		\newblock {\em Total positivity. {V}ol. {I}}.
		\newblock Stanford University Press, Stanford, CA, 1968.
		
		
		
		
%
		
		
%
%

        \bibitem{Lu94}
        G.~Lusztig.
        \newblock Total positivity in reductive groups.
        \newblock In {\em Lie theory and geometry}, volume 123 of {\em Progr. Math.},
        pages 531--568. Birkh\"{a}user, Boston, MA, 1994.

        \bibitem{Marcus_etal}
        M.~Marcus and H.~Minc.
        \newblock On the relation between the determinant and the permanent.
        \href{https://www.semanticscholar.org/paper/On-the-relation-between-the-determinant-and-the-Marcus-Minc/d41b7d48caf66e74c07a2ee73d04c6c264b47719}{\em Illinois J. Math.}, 5:327--332, 1962.
        
        
        \bibitem{Minkowski1900}
        H.~Minkowski.
        \newblock Zur Theorie der Einheiten in den algebraischen Zahlk\"orpern.
        \newblock\href{https://eudml.org/doc/58467}{\em Math.-Phys. Klasse}, 90--93, 1900.
		
		\bibitem{Minkowski1907}
		H.~Minkowski.
		\newblock {\em Diophantische Approximationen.}
		\newblock\href{https://link.springer.com/book/10.1007/978-3-663-16055-7}{Tuebner, Leipzig}, 1907.
		
		\bibitem{Motzkin52}
		T.S.~Motzkin.
		\newblock{Copositive quadratic forms.}
		\newblock{\em National Bureau of Standards Report}, 1818:11--12, 1952. 
		
		\bibitem{Nagy17}
		G.~Nagy.
		\newblock Determinant preserving maps: an infinite dimensional version of a theorem of Frobenius.
		\newblock \href{https://doi.org/10.1080/03081087.2016.1185082}{\em Linear Multilinear Algebra}, 65:351--360, 2017.
		
		\bibitem{Ostrowski37}
		A.~Ostrowski.
		\newblock \"{U}ber die determinanten mit \"{u}berwiegender Hauptdiagonale.
		\newblock\href{https://link.springer.com/article/10.1007/BF01214284}{\em Comment. Math. Helv.}, 10:69--96, 1937.
		
		
		\bibitem{Ostrowski56}
		A.~Ostrowski.
		\newblock Determinanten mit \"{u}berwiegender Hauptdiagonale und die absolute Konvergenz von linearen Iterationsprozessen.
		\newblock\href{https://link.springer.com/article/10.1007/BF02564340}{\em Comment. Math. Helv.}, 30:175--210, 1956.
		
		\bibitem{pinkus}
		A.~Pinkus.
		\newblock {\em Totally positive matrices}.
		\newblock \href{http://dx.doi.org/10.1017/CBO9780511691713}{Cambridge Tracts in Mathematics, Vol.~181}, Cambridge University Press, Cambridge, 2010.
		
		
		
		\bibitem{Pokora}
		P.~Pokora.
		\newblock Introduction to linear preserver problems. Unpublished remarks.
		
		\bibitem{Polya}
		G.~P\'olya.
		\newblock Aufgabe 424.
		\newblock {\em Arch. Math. Phys. Ser. 3}, 20:271, 1913.
		
		
		\bibitem{Postnikov06}
		A.~Postnikov.
		\newblock Total positivity, Grassmannians, and networks.
		\newblock {\em Preprint}, \href{https://arxiv.org/abs/math/0609764}{ arXiv:math.CO/0609764}, 2006.
		
		
		
		\bibitem{CR_Rao75}
		C.R.~Rao.
		\newblock{\em Linear Statistical Inference and its Applications} (2nd ed.).
		\newblock\href{https://epubs.siam.org/doi/10.1137/1017090}{Wiley Series in Probability and Mathematical Statistics}, John Wiley \& Sons, New York-London-Sydney, 1975.
		
		\bibitem{Ri03}
		K.C.~Rietsch.
		\newblock Totally positive {T}oeplitz matrices and quantum cohomology of
		partial flag varieties.
		\newblock \href{http://dx.doi.org/10.1090/S0894-0347-02-00412-5}%
		{\em J.\ Amer.\ Math.\ Soc.}, 16(2):363--392, 2003.
		
		
		
%
%
%
		\bibitem{Schneider65}
		H.~Schneider. 
		\newblock Positive operators and an inertia theorem. 
		\newblock\href{https://doi.org/10.1007/BF01397969}{\em Numer. Math.}, 7:11--17, 1965.
		
		

	\bibitem{Monderer22}
N.~Shaked-Monderer.
\newblock Preservers of the cp-rank.
\newblock\href{https://doi.org/10.1080/03081087.2021.1948496}{\em Linear Multilinear Algebra}, 70:6211--6222, 2022.
		
		\bibitem{Monderer-Berman}
		N.~Shaked-Monderer and A.~Berman.
		\newblock{\em Copositive and completely positive matrices}.
		\newblock\href{https://www.worldscientific.com/worldscibooks/10.1142/11386?srsltid=AfmBOoqe_9-CzENjuoLylPM-pUZ8Ic-BZF601cP1u_Drhjoal68Awc4L#t=aboutBook}{World Scientific}, 2021.
		
	
		
		
		
		
		\bibitem{Shitov21}
		Y.N.~Shitov.
		\newblock Linear mappings preserving the copositive cone.
		\newblock\href{https://doi.org/10.1090/proc/15432}{\em Proc. Amer. Math. Soc.}, 149:3173--3176, 2021.
		
		
		\bibitem{Shitov23}
		Y.N.~Shitov.
		\newblock Linear mappings preserving the completely positive rank.
		\newblock\href{https://doi.org/10.1016/j.ejc.2022.103641}{\em European J. Combin.}, 109, paper No. 103641 (9 pp.), 2023.
		
		\bibitem{Beasley_Song_et_al_16}
		S.-Z.~Song, L.B.~Beasley, P.~Mohindru, and R.~Pereira.
		\newblock Preservers of completely positive matrix rank.
		\newblock\href{https://doi.org/10.1080/03081087.2015.1082960}{\em Linear Multilinear Algebra}, 64:1258--1265, 2016.
		
		
		
		
		
		
		\bibitem{Tam_IM90}
		B.-S.~Tam and P.-H.~Liou.
		\newblock Linear transformations which map the class of inverse M-matrices onto itself.
		\newblock\href{https://doi.org/10.5556/j.tkjm.21.1990.4651}{\em Tamkang J. Math.}, 21:159--167, 1990.
		
		\bibitem{Tan}
		V.~Tan and F.~Wang.
		\newblock On determinant preserver problems.
		\newblock\href{https://www.sciencedirect.com/science/article/pii/S0024379502007395}{\em Linear Algebra Appl.}, 369:311--317, 2003.
		
		\bibitem{Trefethen-Bau}
		L.N.~Trefethen and D.~Bau.
		\newblock{\em Numerical linear algebra}.
		\newblock \href{https://epubs.siam.org/doi/10.1137/1.9780898719574}{SIAM}, 1997.
		
		
	\end{thebibliography}
\end{document}